\documentclass[a4paper, 11pt]{article}
\usepackage{geometry}
\usepackage{amsmath, amsthm, amsfonts, amssymb, stmaryrd}
\usepackage{enumerate, graphicx}
\usepackage[all]{xy}
\usepackage{arydshln}
\newcommand{\CP}{\mathbb{C}P}
\newcommand{\real}{\mathbb{R}}
\newcommand{\lact}{\curvearrowright}
\DeclareMathOperator{\SO}{SO}\DeclareMathOperator{\Orth}{O}
\DeclareMathOperator{\SU}{SU}\DeclareMathOperator{\U}{U}

\DeclareMathOperator{\cn}{cn}\DeclareMathOperator{\sn}{sn}
\DeclareMathOperator{\Vol}{Vol}\DeclareMathOperator{\Ric}{Ric}

\DeclareMathOperator{\tr}{tr}
\DeclareMathOperator{\id}{id}
\DeclareMathOperator{\Isom}{Isom}
\DeclareMathOperator{\Hol}{Hol}

\DeclareMathOperator{\Met}{Met}

\DeclareMathOperator{\Diff}{Diff}
\DeclareMathOperator{\grad}{grad}

\theoremstyle{plain}
\newtheorem{proposition}{Proposition}[section]
\newtheorem{corollary}[proposition]{Corollary}
\newtheorem{lemma}[proposition]{Lemma}
\newtheorem{theorem}[proposition]{Theorem}
\theoremstyle{definition}

\newtheorem{remark}[proposition]{Remark}

\numberwithin{equation}{section}

\usepackage{titling}

\thanksmarkseries{alph}

\author{Nobuhiko Otoba\thanks{
Keio University, 3-14-1 Hiyoshi, Kohoku-ku, Yokohama, Kanagawa 223-8522, Japan
\indent E-mail address: otoba@math.keio.ac.jp}
	\and Jimmy Petean\thanks{CIMAT, A.P. 402, 36000, Guanajuato. Gto., M\'exico
\indent E-mail address: jimmy@cimat.mx}}
\title{Metrics of constant scalar curvature\\ on sphere bundles}
\date{}

\begin{document}

\maketitle

%\tableofcontents

\begin{abstract}
Let $G/H$ be a Riemannian homogeneous space.  
%the round sphere, where $H$ is a closed subgroup of a compact Lie group $G$ acting on $S^m(1)$ by isometries.  
For an orthogonal representation $\phi$ of $H$ on the Euclidean space $\real^{k+1}$, 
there corresponds the vector bundle $E=G\times_{\phi}\real^{k+1} \to G/H$ with fiberwise inner product.  
%Provided that $\phi$ is either irreducible or the direct sum of an irreducible and a $1$-dimensional trivial representations, 
Provided that $\phi$ is the direct sum of at most two representations which are either trivial or irreducible, 
%Under some assumptions on $\phi$, 
we construct metrics of constant scalar curvature on the unit sphere bundle $UE$ of $E$.  
%The construction is reduced to an ODE, which we solve by means of Jacobi elliptic functions.  
When $G/H$ is the round sphere, we study the number of constant scalar curvature metrics 
in the conformal classes of these metrics.  
%
%If $\phi$ is a trivial representation, then the resulting metrics are the Riemannian direct products, 
%and the estimates are reduced to 
\end{abstract}

\vspace{\baselineskip}
\noindent\textbf{Keywords}\quad
The Yamabe problem $\cdot$ 
Constant scalar curvature $\cdot$ 
%Uniqueness and multiplicity of critical points $\cdot$ 
Spectrum of Laplacian $\cdot$ 
Riemannian submersion with totally geodesic fibers $\cdot$ 
Connection metric $\cdot$
%Sasaki metric on the unit sphere bundle $\cdot$ 
Sphere bundle% $\cdot$ 
%Rotationally symmetric spheres $\cdot$ 
%Jacobi elliptic functions

\vspace{\baselineskip}
\noindent\textbf{Mathematics Subject Classification (2010)}\quad 
53C20

\section{Introduction and results}\label{intro}
Every closed Riemannian manifold of dimension $\ge 3$ 
can be conformally deformed to have constant scalar curvature.  
This is a consequence of the affirmative answer to the Yamabe problem (cf. \cite{LP}, \cite{Aub2}), 
the variational problem resolved in the mid 1980s through combined efforts of 
Yamabe \cite{Yam}, Trudinger \cite{Tru}, Aubin \cite{Aub1}, and Schoen (\cite{Sch1}, \cite{SY}).  
For a metric $g$ of constant scalar curvature, 
it is then interesting to ask how many constant scalar curvature metrics of unit volume there are in the conformal class of $g$.  
%up to multiplicative constants.  
%See Sect.\ \ref{sect:number} for related results.  
%For example, if either $g$ has nonpositive scalar curvature 
%or $g$ is an Einstein metric not conformally equivalent to the round metric of the sphere, 
%then $g$ is the unique metric of constant scalar curvature in its conformal class up to multiplicative constants.   
%In this paper, we consider this question when the manifold is a sphere bundle over a sphere.  
In the present article, we construct metrics of constant scalar curvature on sphere bundles over spheres 
and study the number of constant scalar curvature metrics in their conformal classes. 

%We employ the following notation throughout this article.  
Throughout this article, 
$R(g)$ %or $R_g$ 
is the scalar curvature of a Riemannian metric $g$.  
%$\Ric_g$ is the Ricci curvature of $g$.  
Also, let $S^d(\rho) = (S^d, \rho^2 \overset{\circ}{g}_d)$ be the $d$-dimensional round sphere of radius $\rho$, 
where $\overset{\circ}{g}_d$ is the Riemannian metric on $S^d$ of constant sectional curvature $1$.  

Previous study of O.\ Kobayashi (\cite{KobO1}, \cite{KobO2}), R.\ Schoen \cite{Sch2}, 
and Petean \cite{Petea} (see also Jin--Li--Xu \cite{JLX}) 
on the direct product of round spheres implies the following.  
%
%Let $S^d(\rho) = (S^d, \rho^2 \overset{\circ}{g}_d)$ be the $d$-dimensional round sphere of radius $\rho$, 
%where $\overset{\circ}{g}_d$ is the Riemannian metric on $S^d$ of constant sectional curvature $1$.  
%Consider the Riemannian direct product 
%$S^m(1)\times S^k(r) = (S^m\times S^k, \overset{\circ}{g}_m\oplus r^2\overset{\circ}{g}_k)=:(S^m\times S^k, g(r))$.  
%The metric $g(r)$ has constant scalar curvature $R(g(r))=m(m-1)+k(k-1)/r^2$.  
%The following is known for the direct product of round spheres.  

%\begin{theorem}[\cite{KobO1}, \cite{KobO2}, \cite{Sch2}]\label{Thm:prod1}
%Suppose $m \ge 2$, $k=1$.  
%\begin{enumerate}
%\item Assume $(m-1)r^2 \le 1$.  
%	If a metric $g$ is conformal to $g(r)$ and has constant scalar curvature $R(r)$, 
%	then $g=g(r)$.  
%\item If $(m-1)r^2>l^2$, then the conformal class of $g(r)$ contains at least $l+1$ metrics of constant scalar curvature $R(r)$.  
%	The corresponding conformal factors only depend on the $S^1$-variable.  
%\end{enumerate}
%\end{theorem}
%\item Suppose $m, k \ge 2$.  

\begin{theorem}\label{Thm:prod}
Consider the product metric $g(r)=\overset{\circ}{g}_m\oplus r^2\overset{\circ}{g}_k$ on $S^m\times S^k$ 
for $m+k\ge 3$.  
\begin{enumerate}
%\item Assume $(m-1)r^2=k-1$.  
%	If a metric $g$ is conformal to the Einstein metric $g(r)$ and has constant scalar curvature $R(r)$, 
%	then $g=g(r)$.  
\item Assume $(m-1)r^2\le k$.  
	If a metric $g$ conformal to $g(r)$ has constant scalar curvature and has the same volume as $g(r)$, 
	and if moreover the corresponding conformal factor only depends on the $S^k$-variable, 
	then $g=g(r)$.  
%	provided the corresponding conformal factor only depends on the $S^k$-variable.  
\item The same is true for $(k, m, 1/r)$ in place of $(m, k, r)$.  
%\item Assume $(k-1)r^{-2}\le m$.  
%	If a metric $g$ is conformal to $g(r)$ and has constant scalar curvature $R(r)$, 
%	and if moreover the corresponding conformal factor only depends on the $S^m$-variable, 
%	then $g=g(r)$.  
%\item If $k(k-1)r^{-2}>l(l+m-1)(m+k-1)-m(m-1)$, 
%	then the conformal class of $g(r)$ contains at least $l+1$ metrics of constant scalar curvature $R(r)$.  
%	The corresponding conformal factors only depend on the $S^m$-variable.  
\item If $m(m-1)r^2>l(l+k-1)(m+k-1)-k(k-1)$, 
	then the conformal class of $g(r)$ contains at least $l+1$ metrics of constant scalar curvature $R(g(r))$ 
	such that the corresponding conformal factors only depend on the $S^k$-variable.  
\item The same is true for $(k, m, 1/r)$ in place of $(m, k, r)$.  
\end{enumerate}
\end{theorem}
%\noindent 
We should note that Kobayashi and Schoen succeeded in describing 
the space of all unit-volume metrics of constant scalar curvature in the conformal class of $g(r)$ provided $m=1$ or $k=1$.  
More precisely, if $m=1$ or $k=1$, then the metrics appearing in Theorem \ref{Thm:prod} 
are the only metrics of constant scalar curvature in the conformal class of $g(r)$ up to multiplicative constants and isometries.  

The conformal factors in 3 (or 4) of Theorem \ref{Thm:prod} are radial 
in a sense that they are invariant under the cohomogeneity-one action $\SO(k)\lact S^k$ (or $\SO(m)\lact S^m$) fixing exactly two points.  
%\cite{Petea} \cite{JLX}
Henry--Petean \cite{HP} found non-radial solutions using isoparametric hypersurfaces.  

Piccione et al.\ deal with other situations than $S^m(1)\times S^k(r)$, 
such as Riemannian direct products \cite{LPZ1}, 
Hopf fibrations \cite{BP1}, 
and homogeneous Riemannian submersions \cite{BP2}.  %, 
%and derive weaker conclusions than Theorem \ref{Thm:prod}.  

In this article, we attempt to draw a picture similar to Theorem \ref{Thm:prod} for non-trivial sphere bundles over spheres.  
Let $(B, \check{g})=G/H$ be a Riemannian homogenous space.  
Here, $G$ is a connected Lie group acting transitively on $(B, \check{g})$ by isometries, 
and $H$ is the isotropy subgroup at a point of $B$.  
For an orthogonal representation $\phi: H \to SO(k+1)$ of $H$ on the Euclidean space $\real^{k+1}$, 
there corresponds the vector bundle $E=G\times_{\phi}\real^{k+1} \to G/H$ with fiberwise inner product $\langle \cdot, \cdot \rangle$.  
%$E$ is the quotient of $G\times \real^{k+1}$ with respect to the diagonal action 
Let $UE=\{s\in E \mid \langle s, s\rangle =1\}$ be the unit sphere bundle of $(E, \langle \cdot, \cdot \rangle)$.  

If $\phi$ is the trivial representation on $\real^{k+1}$, 
then $E=(G/H) \times \real^{k+1}$ and $UE=(G/H) \times S^k(1)$ are trivial bundles.  
Since the scalar curvature of $\check{g}$ is constant, 
the product metric $g(1):=\check{g}\oplus \overset{\circ}{g}_k$ on $UE$ has constant scalar curvature.  
Rescaling the fiber, we obtain the product metric $g(r):=\check{g}\oplus r^2\overset{\circ}{g}_k$ on $UE$ 
of constant scalar curvature 
%\begin{align*}
$R(\check{g})+k(k-1)/r^2$ for $r>0$.  
%\end{align*}
%
%
%
%Let $\phi$ is an arbitrary representation.  
%The Lie algebra $\mathfrak{g}$ of $G$ has the $\Ad(H)$-invariant inner product 
%corresponding to the $G$-invariant Riemannian metric $\check{g}$ on $G/H$.  
%Let $\mathfrak{m}=\mathfrak{h}^{\perp}$ be the orthogonal complement of the Lie algebra $\mathfrak{h}$ of $H$ 
%with respect to this inner product.  
%The $\Ad(H)$-invariant complement $\mathfrak{m}$ to $\mathfrak{h}$ defines the principal connection $\omega$ %on $H\to G\to G/H$.  
%Let $\nabla$ be the metric connection on $E$ associated with $\omega$.  
%%
%Sasaki metrics and connection metrics
%
%We will see in particular that the situation is similar for the bundles obtained by irreducible representations (over a round sphere). 
%We first show the existence of a family of constant scalar curvature metrics.  
%
For irreducible representations, we prove: 
\begin{theorem}\label{Thm:irrep}
Let $(B, \check{g})=G/H$ be a Riemannian homogenous space and 
$\phi: H \to SO(k+1)$ an orthogonal representation of $H$ on $\real^{k+1}$.  
Assume $\phi$ is irreducible.  
Then, for every $r>0$, 
there exists a Riemannian metric $g(r)$ on the unit sphere bundle $UE$ of $E=G\times_{\phi}\real^{k+1}$ such that \begin{itemize}
\item the projection $(UE, g(r))\to (B, \check{g})$ is a Riemannian submersion 
	with totally geodesic fibers each of which is isometric to the round sphere $S^k(r)$, and 
\item $g(r)$ has constant scalar curvature $R(\check{g})+k(k-1)/r^2 - a^2r^2$.  
\end{itemize}
%\begin{align*}
%\check{R}+k(k-1)/r^2 - a^2r^2.  
%\end{align*}
Here, $a\ge 0$ is a real number independent of $r$.  %, and 
%$a=0$ if %and only if 
%%the metric $\check{g}$ is flat or if 
%the representation $\phi$ is trivial.  
%%the unit sphere bundle $UE$ of the vector bundle $E=G\times_{\phi}\real^k$ admit metrics $g(R)$ of constant scalar curvature 
\end{theorem}
\noindent 
For the proof, we look at the isotropy group as in Boeckx--Vanhecke \cite{BL}.  
Guijarro--Sadun--Walschap \cite[Proposition 5.2]{GSW} looked at the holonomy group and obtained the same result 
under an additional assumption that $E$ admits a parallel connection.  %(cf. \cite[p.266]{GSW}).  

%We look at isotropy group instead of holonomy group as .  
%all Milnor spheres.  
%It turns out that we have to 

%\subsection{The number of constant scalar curvature metrics}

%If $G/H$ be isometric to the round sphere, we have the following.  
If $G/H$ is isometric to the round sphere, 
we can study the number of constant scalar curvature metrics in the
conformal class of $g(r)$ as in Theorem \ref{Thm:prod}.  
\begin{theorem}\label{Thm:number}
Suppose $S^m(1)=G/H$, and 
let $g(r)$ be the metric in Theorem \ref{Thm:irrep} for $m+k\ge 3$.  %.  and $R(r)$ its scalar curvature.  
\begin{enumerate}
%\item Assume $m(m-1)+k(k-1)/r^2-a^2r^2\le 0$.  
%	If a metric $g$ is conformal to $g(r)$ and has constant scalar curvature $R(r)$, then $g=g(r)$.  
\item Assume $-(a^2/m)r^4+(m-1)r^2\le k$.  
	If a metric $g$ conformal to $g(r)$ has constant scalar curvature and has the same volume as $g(r)$, 
	and if moreover the gradient vector field of the corresponding conformal factor is tangent to the fibers, 
	then $g=g(r)$.  
%	provided the gradient vector field of the corresponding conformal factor is tangent to the fibers.  
\item Assume $-(a^2/k)r^2+(k-1)/r^2\le m$.  
	If a metric $g$ conformal to $g(r)$ has constant scalar curvature and has the same volume as $g(r)$, 
	and if moreover the corresponding conformal factor is constant along each fiber, 
	then $g=g(r)$.  
\item If $-a^2r^2+k(k-1)/r^2>l(l+m-1)(m+k-1)-m(m-1)$, 
	then the conformal class of $g(r)$ contains at least $l+1$ metrics of constant scalar curvature $R(g(r))$ 
	such that the corresponding conformal factors are constant along each fiber.  
\end{enumerate}
\end{theorem}
\noindent 
%When the representation $\phi$ is trivial, 
%then $a=0$ and we recover the statements 1, 2, and 4 of Theorem \ref{Thm:prod}.  
%We note that 
The statement corresponding to (3) of Theorem \ref{Thm:prod} is not necessarily true for irreducible representations 
%In fact, when the representation $\phi$ on $\real^{k+1}$ is transitive on $S^k(1)$, there holds:   
%If the gradient vector field of a function on $UE$ is tangent to the fibers, then it already has to be constant.  
(cf.\ Sect.\ \ref{sect:number}).  

We can also build metrics of constant scalar curvature in the following situation.  
%Every orthogonal representation $\phi: H\to SO(k+1)$ of a compact group $H$ is completely reducible.  
%Namely, $\phi$ is equivalent to the direct sum 
%$\mathbf{1}\oplus \phi_1\oplus \phi_2 \oplus \dots \oplus \phi_l$ of 
%the trivial representation $\mathbf{1}$ on the Euclidean space of dimension $k_{0}\ge 0$ 
%and $l\ge 0$ irreducible representations $\phi_1, \phi_2, \dots, \phi_l$.  
\begin{theorem}\label{Thm:sum}
Let $(B, \check{g})=G/H$ be a Riemannian homogenous space and 
$\phi: H \to SO(k+1)$ an orthogonal representation of $H$ on $\real^{k+1}$.  
Assume $\phi$ is the direct sum of at most two representations which are either trivial or irreducible.  
%Assume, in the notation above, 
%\begin{enumerate}
%\item $k_0\ge 1$, $l\le 1$, 
%\item $k_0=0$, $l\le 2$, or
%\item $\phi_1=\phi_2=\dots=\phi_l$.  
%\end{enumerate} 

Then, there exist one or two one-parameter families of Riemannian metrics with constant scalar curvature 
on the unit sphere bundle $UE$ of $E=G\times_{\phi}\real^{k+1}$ 
such that the projection $UE\to (B, \check{g})$ becomes Riemannian submersions with totally geodesic fibers.  
The (intrinsic) scalar curvatures of the typical fibers are not necessarily constant.  
\end{theorem}
\noindent 
%The representation $\phi$ is trivial if $l=0$ 
%and is irreducible if $k_{0}=0$ and $l=1$, 
%in which cases 
%If the representation $\phi$ is trivial or irreducible, then we have Theorem \ref{Thm:irrep}.  
When $G/H=\SU(2)/\U(1)=\CP^1$ and $\phi$ is the direct sum of an irreducible and the $1$-dimensional trivial representations, 
then we recover the constant scalar curvature metrics on Hirzebruch surfaces in \cite{Oto}.  
We do not know if the statement of Theorem \ref{Thm:sum} holds 
%without assumptions 1, 2, or 3.  
for every orthogonal representation $\phi: H\to \SO(k+1)$.  
%In Sect.\ \ref{sect:number}, we consider the number of constant scalar curvature metrics in the conformal classes 
%of the metrics in Theorem \ref{Thm:sum}.  
Theorem \ref{Thm:sum} is a consequence of Theorem \ref{Thm:irrep} and the following
\begin{theorem}\label{Thm:FiberJoin}
Let $E_i \to B$ be a real vector bundle of rank $k_i+1\ge 1$, 
$\langle \cdot, \cdot \rangle_i$ a fiberwise inner product on $E_i$, 
and $\nabla^i$ a metric connection on $\left( E_i, \langle \cdot, \cdot \rangle_i \right)$ for $i=1, 2$.  
Take a Riemannian metric $\check{g}$ on $B$ 
and define the connection metric $g_i=\check{g} \oplus \overset{\circ}{g}_{k_i}$ on $\pi_i: UE_i\to B$ 
with respect to the Ehresmann connection induced by $\nabla^i$.  
Assume the scalar curvatures of $\check{g}$, $g_1$, and $g_2$ are constant.  

Then, %for $R$ large enough, 
there exist one or two one-parameter families of connection metrics with constant scalar curvature %$R$ 
on the unit sphere bundle $\pi: U(E_1\oplus E_2)\to (B, \check{g})$.  %of the Whitney sum $E_1\oplus E_2$.    
The (intrinsic) scalar curvatures of the typical fibers are not necessarily constant.  
%
%For a Riemannian metric $\check{g}$ on $B$, 
%define the connection metric $g_i=\check{g} \oplus \overset{\circ}{g}_{k_i}$ on the unit sphere bundle $UE_i$ % of $E_i$ 
\end{theorem}
\noindent 
Some remarks on Theorem \ref{Thm:FiberJoin} are in order.  
(1) By definition of connection metrics, 
the projection $U(E_1\oplus E_2)\to (B, \check{g})$ becomes Riemannian submersions with totally geodesic fibers (see Sect.\ \ref{sect:ConnMet}).  
(2) If neither $\nabla^1$ nor $\nabla^2$ is flat, there are two one-parameter families; otherwise there is only one (see Remark \ref{rem:parameter}).  
(3) The typical fibers of these connection metrics can have constant scalar curvature 
only if %$\lvert A^1\rvert=\lvert A^2\rvert$ 
the norms of O'Neill's integrability tensors for $\pi_1$ and $\pi_2$ are equal 
(see Remark \ref{rem:inductive}).  

%\noindent The fibers of the metrics in Theorem \ref{Thm:sum} do not necessarily have constant scalar curvature.  
%When the fibers do not have constant scalar curvature, 
%rescaling of the fiber does not produce other metrics of constant scalar curvature.  
%See Sect.\ ?? for details.  
%$\phi$ is the direct sum of at most two representations which are either trivial or irreducible. 

%It is interesting to ask if Theorem ?? holds without any assumptions on the representation $\phi$.  

%It is interesting to consider the following situation.  
%Let $\phi: H\to \SO(k+1)$ be an orthogonal irreducible representation 
%and $\mathbf{1}: H \to \real^{\times}$ the ($1$-dimensional) trivial representation.  
%Consider the direct sum representation $\phi\oplus \mathbf{1}: H\to \SO(k+2)$ on $\real^{k+2}$.  
%
%\begin{theorem}Assume $\omega$ is not flat.  
%For each real number $R$, there exists a Riemannian metric $g(R)$ 
%on the unit sphere bundle $UE$ of $E=G\times_{\phi\oplus\mathbf{1}}\real^{k+2}$ such that 
%\begin{enumerate}
%\item $g(R)$ has constant scalar curvature $R$, 
%\item the projection $\pi: (UE, g(R)) \to G/H$ is a Riemannian submersion with totally geodesic fibers, 
%\item yet the fibers of $\pi$ do not have constant scalar curvature.  
%\end{enumerate}
%\end{theorem}

%It is interesting to ask if one can construct metrics of constant scalar curvature 
%for an arbitrary representation $\phi$.  

This paper is organized as follows.  
In Sect.\ \ref{sect:ConnMet}, %we fix notation for Riemannian submersions with totally geodesic fibers.  
we recall the notion of connection metrics.  
In Sect.\ \ref{sect:construction}, we construct Riemannian metrics of constant scalar curvature on sphere bundles.  
After proving Theorem \ref{Thm:irrep} in Sect.\ \ref{sect:proofirrep}, 
we introduce the fiberwise join of two sphere bundles and prove Theorem \ref{Thm:FiberJoin} in Sect.\ \ref{sect:FiberJoin}.  
Theorem \ref{Thm:sum} then follows easily (Sect.\ \ref{sect:ProofOfThm:sum}).  
%over Riemannian homogeneous spaces and prove Theorems \ref{Thm:irrep}, \ref{Thm:FiberJoin}, and \ref{Thm:sum}.  
In Sect.\ \ref{sect:number}, 
we prove Theorem \ref{Thm:number} after discussing the Yamabe problem and Riemannian submersion with totally geodesic fibers.  
%The contents of Sects. \ref{sect:construction} and \ref{sect:number} are logically independent.  
%Putting these results together, we prove Theorem \ref{Thm:number} in Sect.\ \ref{sect:proof}.  

\section{Connection metrics}\label{sect:ConnMet} %and Riemannian submersion with totally geodesic fibers}

Let $\pi: M \to B$ be a product bundle.  $M=B\times F$ and $\pi$ is the projection onto the first factor.  
For Riemannian metrics $\check{g}$ and $\hat{g}$ on $B$ and $F$, respectively, 
we can define the product metric $g:= \check{g} \oplus \hat{g}$ on the total space $M$.  
This procedure was generalized by Vilms to fiber bundles with structure group, 
and the resulting metrics are called connection metrics in modern terminology \cite[\S 2.7]{GW}.  
We set up the notation in this section.  

We recall the result of Hermann \cite{Herm}.  
Let $\pi: (M, g)\to (B, \check{g})$ be a Riemannian submersion with totally geodesic fibers 
and assume $g$ is complete.  
We call $VM=\ker \pi_*$ and $HM=VM^{\perp g}$ the vertical and horizontal subbundles of $TM$, respectively.  
Let $\check{\gamma}: [0, 1]\to B$ be a path.  
For every choice of $x\in \pi^{-1}(\check{\gamma}(0))$, 
we can uniquely lift $\check{\gamma}$ to the path $\gamma: [0, 1]\to M$ 
starting from $x$ so that the velocity vector field $\frac{d}{dt}\gamma$ is horizontal.  
Moreover, the parallel transport $\pi^{-1}(\check{\gamma}(0)) \to \pi^{-1}(\check{\gamma}(1))$ is an isometry.  
Therefore, 
\begin{enumerate}
\item All fibers of $\pi$ are mutually isometric.  
We choose a base point $o$ of $B$, call $\pi^{-1}(o)=(F, \hat{g})$ the typical fiber, 
and denote by $\iota: F\to M$ the inclusion map.  
\item We can introduce a Lie group $H$ acting on $(F, \hat{g})$ isometrically so that the following holds.  
$\pi: M \to B$ is a fiber bundle with structure group $H$, 
and the holonomy group at $o$ of the Ehresmann connection $HM$ is contained in the image of the action $H\to \Isom(F, \hat{g})$.  
\end{enumerate}
We summarize this situation as 
\begin{align}\label{eq:TGF}
H\overset{\text{isom}}{\lact}(F, \hat{g}) \xrightarrow{\iota} (M, g) \xrightarrow{\pi} (B, \check{g}).  
\end{align}

Vilms \cite{Vil} considered the converse.  
Let $H\lact F \xrightarrow{\iota} M \xrightarrow{\pi} B$ 
be a fiber bundle with structure group and $HM\subset TM$ an Ehresmann $H$-connection.  
Here, $H$ is a (finite-dimensional) Lie group acting on $F$, 
$F=\pi^{-1}(o)$ for some $o\in B$, 
and the holonomy group at $o$ of $HM$ is contained in the image of the action $H \to \Diff(F)$.  
Assume that there exists a Riemannian metric on $F$ invariant under the action of $H$ and that $M$ is connected.  
For a metric $\check{g}$ on $B$ and an $H$-invariant metric $\hat{g}$ on $F$, 
we can uniquely define a Riemannian metric $g:=\check{g} \oplus_{HM} \hat{g}$ on $M$ so that 
\begin{enumerate}
\item $\pi: (M, g) \to (B, \check{g})$ is a Riemannian submersion and $(\ker \pi_*)^{\perp g} = HM$, 
\item the parallel transport $\pi^{-1}\left( \check{\gamma}(0)\right) \to \pi^{-1}\left( \check{\gamma}(1)\right)$ 
with respect to $HM$ is an isometry for every path $\check{\gamma}: [0, 1] \to B$, and 
\item the inclusion map $\iota: F \to M$ is isometric:  $\hat{g}=\iota^* g$.   
\end{enumerate}  
With the metric $g$, each fiber of $\pi: (M, g) \to (B, \check{g})$ is totally geodesic.  
If both $\check{g}$ and $\hat{g}$ are complete, then $g$ is complete, 
and we recover the situation \eqref{eq:TGF}.  
We call $g=\check{g} \oplus_{HM} \hat{g}$ the connection metric of $\check{g}$ and $\hat{g}$ with respect to the Ehresmann connection $HM$.  
If $\pi: M=B\times F\to B$ is the product bundle, 
then the product metric $\check{g}\oplus\hat{g}$ agrees with the connection metric $\check{g}\oplus_{HM}\hat{g}$, 
where  $HM$ %$HM=\pi^*TB$ 
is the integrable Ehresmann connection.  
We usually write a connection metric simply as $g=\check{g} \oplus_{HM} \hat{g}=\check{g} \oplus \hat{g}$ 
whenever the Ehresmann connecion $HM$ is clear from the context.  
%See \cite[\S 2.7]{GW}.  
%$g=\check{g}\oplus\hat{g}$ 
%

In what follows, we always assume that the total space of a Riemannian submersion with totally geodesic fibers is complete 
and connected.  % that the base space is connected.  

\section{Construction of metrics}\label{sect:construction}

\subsection{Integrability of connections}
%{The integrability tensor of O'Neill}

Let 
%$G\overset{\text{isom}}{\lact}(F, \hat{g})\to%\xrightarrow{\iota} 
%(M, g)\xrightarrow{\pi} (B, \check{g})$ 
$\pi: (M, g)\to (B, \check{g})$ 
be a Riemannian submersion with totally geodesic fibers.  
The scalar curvature of $g$ satisfies 
\begin{align}\label{eq:ONeill}
%R_g(x)=R_{\check{g}}(\pi(x))+R_{\pi^{-1}(\pi(x))}(x)-\lvert A\rvert^2(x).  &&(x\in M) 
R(g)=\pi^*R(\check{g})+\hat{R}-\lvert A\rvert^2.  
\end{align}
Here, $\hat{R}$ is the scalar curvature of the fibers, 
and $\lvert A\rvert$ is the norm of O'Neill's integrability tensor.  
More precisely, 
$\hat{R}(x)$ is the scalar curvature at $x\in M$ of the Riemannian submanifold $\pi^{-1}(\pi(x))\subset (M, g)$, 
and %we adopt the convention of Besse \cite[(9.33a)]{Bes3} so that 
\begin{align}\label{eq:DefONeill}
\lvert A\rvert^2=\sum_{i, j=1}^mg(A_{E_i}E_j, A_{E_i}E_j), &&
A_{E_i} {E_j}=\frac{1}{2}\mathcal{V}[E_i, E_j]
\end{align}
where $m=\dim B$, 
$E_1, \dots, E_m$ is an orthonormal basis of $HM=\left( \ker \pi_*\right)^{\perp}$, 
and $\mathcal{V}$ is the projection onto $VM=\ker \pi_*$.  
See \cite[Chapter 9]{Bes3}.  
%and $A_{E_i} {E_j}=\frac{1}{2}\mathcal{V}[E_i, E_j]$.  
%O'Neill \cite{O'N} defines two fundamental tensors $T$ and $A$ for a Riemannian submersion 
%(cf.\ Gray \cite{Gra}, Karcher \cite{Kar}, and \cite[Chapter 9]{Bes3}, \cite[Chapter 1]{FIP}).  
%The tensor $T$ or $A$ vanishes if and only if, respectively, 
%every fiber is totally geodesic or the connection (i.e. horizontal distribution) is integrable.  
%The norm $\lvert A\rvert$ of the integrability tensor plays an important role 
%in the understanding of scalar curvature.  
%
%
%\begin{notation}
%We adopt the convention of Besse \cite[(9.33a), 9.37]{Bes3} so that 
%\[
%\lvert A\rvert^2=\sum_{i, j=1}^mg(A_{E_i}E_j, A_{E_i}E_j), 
%\]
%which is equal to the usual squared tensor norm 
%$
%\sum_{i, j, k=1}^n\left( g(A_{E_i}E_j, E_k)\right)^2
%$
%of $A$ up to multiplicative constant.  
%Here, $E_1, E_2, \dots, E_n$ consist an orthonormal moving frame of $TM$ 
%so that $E_1, E_2, \dots, E_m$ are horizontal.  
%\end{notation}
$\lvert A\rvert=0$ on $M$ if and only if the horizontal distribution $HM$ is integrable.  

\begin{lemma}\label{HhatR}
Let $\pi: (M, g)\to (B, \check{g})$ be a Riemannian submersion with totally geodesic fibers.  
If $\check{H}$ is a vector field on $B$ and $H$ its horizontal lift to $M$, 
then $H(\hat{R})=0$.  
\end{lemma}

\begin{proof}
Let $\gamma$ be an integral curve of $H$.  
$\gamma$ is a horizontal lift of an integral curve $\check{\gamma}$ of $\check{H}$.  
Since the parallel transport of fibers along $\check{\gamma}$ is an isometry, 
$\hat{R}$ is constant along $\gamma$.  
Hence $H(\hat{R})=0$.  
\end{proof}

%\begin{proposition}\label{cscRiemSubm0}
%Let $(F, \hat{g})\xrightarrow{\iota} (M, g)\xrightarrow{\pi} (B, \check{g})$ be a Riemannian submersion with totally geodesic fibers, 
%and assume $g$ has constant scalar curvature.  
%Then, the norm $\lvert A\rvert$ of O'Neill tensor for $\pi$ is constant 
%if and only if both $\check{g}$ and $\hat{g}$ have constant scalar curvature.  
%\end{proposition}
\begin{proposition}\label{cscRiemSubm}
Let $(F, \hat{g})\to%\xrightarrow{\iota} 
(M, g)\xrightarrow{\pi} (B, \check{g})$ be a Riemannian submersion with totally geodesic fibers, 
and assume $g$ has constant scalar curvature.  

Then, the norm $\lvert A\rvert$ of O'Neill tensor for $\pi$ is constant 
if and only if both $\check{g}$ and $\hat{g}$ have constant scalar curvature.  
Moreover, the following dichotomy holds.  
%Moreover, we have the following dichotomy.  
\begin{enumerate}
\item $\lvert A \rvert$ is constant if and only if $\check{g}\oplus c\hat{g}$ 
has constant scalar curvature for every real number $c>0$.  
\item $\lvert A \rvert$ is not constant if and only if 
the scalar curvature of $\check{g}\oplus c\hat{g}$ is not constant for every real number such that $c\neq 1$, $c>0$.    
\end{enumerate}
Here, $\check{g}\oplus c\hat{g}$ is the connection metric with respect to $HM=\left( \ker \pi_*\right)^{\perp g}$.  
\end{proposition}
%\noindent Here, $\check{g}\oplus c\hat{g}$ is the connection metric 
%%of $\check{g}$ and $c\hat{g}$ 
%with respect to the same connection as $g$.  

\begin{proof}
%In this proof, we write $R_g=R(g)$, $R_{\check{g}}=R(\check{g})$ for notational simplicity.  
%
The following observation is convenient.  
For a vector field $\check{H}$ on $B$ and its horizontal lift $H$ to $M$, 
differentiate the both sides of \eqref{eq:ONeill} by $H$ to get 
$H\left(R(g)\right)=\pi^*\check{H}(R (\check{g}))+H(\hat{R})-H(\lvert A\rvert^2)$.  
Since $R(g)$ is constant, Lemma \ref{HhatR} implies
\begin{equation}\label{eq:DiffH}
H\left(\lvert A\rvert^2\right)=\pi^*\check{H}(R(\check{g})).  
\end{equation}

Suppose $\lvert A\rvert$ is constant.  
Then, \eqref{eq:DiffH} implies $\check{H}(R(\check{g}))=0$ for every vector field $\check{H}$ on $B$.  
Hence $R(\check{g})$ is constant, and $\hat{R}=R(g)-\pi^*R(\check{g})+\lvert A\rvert^2$ is also constant.  
The opposite implication is immediate from \eqref{eq:ONeill}.  
  
For a real number $c>0$, \eqref{eq:ONeill} implies that the scalar curvature of $g_c := \check{g} \oplus c \hat{g}$ is 
\begin{align}\label{eq:ONeill'}
R(g_c)=\pi^*R(\check{g})+c^{-1}\hat{R}-c\lvert A\rvert^2.  
\end{align}
Suppose $\lvert A\rvert$ is constant.  
Since both $R(\check{g})$ and $\hat{R}$ are constant, 
$R(g_c)$ is constant for every $c>0$.  
%\begin{align*}
%&\lvert A\rvert \ \text{is constant}\\
%&\Rightarrow \check{g}+c\hat{g} \ \text{has constant scalar curvature for every} \ c>0\\
%&\Rightarrow \check{g}+c\hat{g} \ \text{has constant scalar curvature for some} \ c\neq 1.  
%\end{align*}
%
Lastly, we claim: 
\begin{equation}\label{eq:claim}
\text{If $R(g_c)$ is constant for some $c\neq 1$, then $\lvert A\rvert$ is constant.  }
\end{equation}
Subtracting $R(g)$ from $R(g_c)$, we obtain 
\begin{equation}\label{eq:aux}
R(g_c)-R(g)=\left( c^{-1}-1\right)\hat{R}-(c-1)\lvert A\rvert^2.  
\end{equation}
Let $\check{H}$ be a vector field on $B$ and $H$ its horizontal lift to $M$.  
Since both $R(g_c)$ and $R(g)$ are constant, 
differentiation of the both sides of \eqref{eq:aux} by $H$ yields $0=-(c-1)H\left(\lvert A\rvert^2\right)$.  
The assumption $c\neq 1$ and \eqref{eq:DiffH} imply %and $H\left( \pi^*R_{\check{g}}\right) = H\left(\lvert A\rvert^2\right)$, 
$0=H\left(\lvert A\rvert^2\right)=\pi^* \check{H}(R(\check{g}))$.  
Since $\check{H}$ is arbitrary, $R(\check{g})$ is constant.  
Multiply the both sides of \eqref{eq:ONeill} by $c$ and subtract the resulting equation from \eqref{eq:ONeill'} to get 
\[
R(g_c) - c R(g) = (1-c) \pi^*R(\check{g}) + (c^{-1}-c)\hat{R}.  
\]
%Since $R(g_c)$, $R_1$, and $R_{\check{g}}$ are constant and 
$\hat{R}$ is constant since $c^{-1}-c\neq 0$.  
This proves \eqref{eq:claim}.  
 
%Differentiation of the both sides with respect to an arbitrary basic vector field $H$ then yields $H(\lvert A\rvert^2)=0$, 
%which in turn implies constancy of $\check{R}$ in view of \eqref{O'NScalar}.  
%Multiply the both sides of \eqref{O'NScalar} by $t$ and subtract the resulting equation from $R_t$ to get 
%\[
%R_t-tR=\left( 1-t\right)\check{R}+\left( t^{-1}-t\right)\hat{R}, 
%\]
%which directly states that $\hat{R}$ is constant.  
%Thus $\lvert A\rvert$ has to be constant if $R_t$ is constant for a single $t\neq 1$.  
\end{proof}

%Let $\pi_i: (M_i, g_i) \to (B_i, \check{g}_i)$ be a Riemannian submersion with totally geodesic fibers for $i=1, 2$.  
%We call a smooth map $\Psi: M_1 \to M_2$ a morphism from $\pi_1$ to $\pi_2$ if 
%\begin{enumerate}
%\item there exists a map $\psi: B_1 \to B_2$ such that $\pi_2 \circ \Psi = \psi \circ \pi_1$, and 
%\item $\Psi^*g_2=g_1$.  
%\end{enumerate}
%Such a map $\psi$ is unique, is smooth, and satisfies $\psi^*\check{g}_2=\check{g}_1$.  

Let $\pi: (M, g) \to (B, \check{g})$ be a Riemannian submersion.  %with totally geodesic fibers.  
If $\Psi \in \Isom(M, g)$ and $\psi\in \Isom (B, \check{g})$ are isometries satisfying $\pi\circ\Psi=\psi\circ\pi$, 
then we call the pair $(\Psi, \psi)$ an automorphism of $\pi$.  
\begin{lemma}\label{lem:AutChar}
Let $\pi: (M, g) \to (B, \check{g})$ be a Riemannian submersion.  %with totally geodesic fibers.  
%We denote by $\mathcal{V}: TM\to TM$ the projection onto $VM=\ker d\pi$ 
%and by $\mathcal{H}=\id - \mathcal{V}$ the projection onto $HM=VM^{\perp}$.  
We denote by $\mathcal{V}, \mathcal{H}: TM\to TM$ the projections onto $VM$, $HM$, respectively.  

If $(\Psi, \psi)$ is an automorphism of $\pi$, then 
\begin{align}\label{eq:HorVerComm}
\Psi_*\circ \mathcal{V} = \mathcal{V} \circ \Psi_*, &&\Psi_*\circ \mathcal{H} = \mathcal{H} \circ \Psi_*.  
\end{align}
In particular, $\Psi_*$ preserves $HM$ and $VM$.  
%$\Psi_*$ maps horizontal and vertical vectors to horizontal and vertical vectors, respectively.  

Conversely, if $\Psi: M\to M$ and $\psi: B\to B$ are diffeomorphisms satisfying $\pi\circ\Psi=\psi\circ\pi$, 
then $(\Psi, \psi)$ is an automorphism of $\pi$ if % the Riemannian submersion $\pi$ if 
$\psi^*\check{g}=\check{g}$, \eqref{eq:HorVerComm} holds, and $g(\Psi_*v, \Psi_*v)=g(v, v)$ for all $v\in VM$.  
\end{lemma} 
%If $\Psi: M \to M$ and $\psi: B \to B$ are diffeomorphisms satisfying $\pi\circ\Psi=\psi\circ\pi$, $\Psi^*= for a map $\psi: B\to B$, 
%If $\Psi: M \to M$ is a diffeomorphism satisfying $\pi\circ\Psi=\psi\circ\pi$ for a map $\psi: B\to B$, 

%\begin{proof}
%The relation $\pi\circ\Psi=\psi\circ\pi$ implies 
%\begin{equation}\label{eq:Commute}
%\pi_* \circ \Psi_* = \psi_* \circ \pi_*.  
%\end{equation}
%In particular, $\pi_*  \left(\Psi_* (v)\right)= 0$ holds for all vertical vectors $v \in VM$, 
%so $\Psi_*$ preserves $VM$.  
%Since $\mathcal{V}(v)=v$ for $v\in VM$, %$\Psi_*\circ \mathcal{V} = \mathcal{V} \circ \Psi_*$ 
%\eqref{eq:HorVerComm} holds on $VM$.  

%Let $v\in HM$ and $\check{v} \in TB$ such that $\pi_*(v)=\check{v}$.  
%\eqref{eq:Commute} implies $\Psi_*(v)$ is a lift of $\psi_*(\check{v})$ 

%We only show the converse.  
%Let $v\in TM$ and write $v=\mathcal{H}(v)+\mathcal{V}(v)$.  
%\eqref{eq:HorVerComm} implies $\Psi_*(v)=$
%\[
%
%\]
%\end{proof}

%Let $\pi: (M, g) \to (B, \check{g})$ be a Riemannian submersion with totally geodesic fibers.  
%We call $\Psi \in \Isom(M, g)$ an automorphism of $\pi$ if 
%$\pi\circ\Psi=\psi\circ\pi$ holds for some $\psi \in \Isom (B, \check{g})$.  
%$\Aut(\pi)$ denotes the collection of all automorphisms of $\pi$.  
%of a Riemannian submersion $\pi: (M, g) \to (B, \check{g})$ with totally geodesic fibers 
%is an isometry  such that 
%holds for an isometry $\psi: (B, \check{g}) \to (B, \check{g})$.  
\begin{lemma}\label{lem:orbitconst}
Let $\pi: (M, g) \to (B, \check{g})$ be a Riemannian submersion %with totally geodesic fibers 
and $(\Psi, \psi)$ an automorphism of $\pi$.  
Then, the norm of O'Neill tensor for $\pi$ satisfies $\Psi^*\lvert A\rvert=\lvert A\rvert$.  
%The norm $\lvert A\rvert$ of O'Neill tensor is constant on each orbit of $\Aut (\pi)$.  
\end{lemma}

\begin{proof}

Let $X_1$, $X_2$ be horizontal vector fields on $M$.  
%$\Psi \in \Aut(\pi)$, 
Set $Y_1=\Psi_*X_1$, $Y_2=\Psi_*X_2$.  
The first relation in \eqref{eq:HorVerComm} implies 
\begin{align*}
2A_{Y_1}{Y_2}
&=2A_{\Psi_*X_1}\Psi_*X_2
=\mathcal{V}[\Psi_*X_1, \Psi_*X_2]
=\mathcal{V}\left(\Psi_*[X_1, X_2]\right)
=\Psi_*\left(\mathcal{V}[X_1, X_2]\right)\\
& =2\Psi_*\left(A_{X_1}X_2\right).  
\end{align*}
At points $x, y\in M$ such that $y=\Psi(x)$, it follows
%\begin{align}\label{eq:AutAux}
%\begin{split}
%g(A_{Y_1(y)}Y_2(y), A_{Y_1(y)}Y_2(y))
%&=g\left(\Psi_*(A_{X_1(x)}X_2(x)), \Psi_*(A_{X_1(x)}X_2(x))\right)\\
%&=g\left(A_{X_1(x)}X_2(x), A_{X_1(x)}X_2(x)\right).  
%\end{split}
%\end{align}
\begin{align}\label{eq:AutAux}
\begin{split}
g_y(A_{Y_1}Y_2, A_{Y_1}Y_2)
&=g_y\left(\Psi_*(A_{X_1}X_2), \Psi_*(A_{X_1}X_2)\right)\\
&=g_x\left(A_{X_1}X_2, A_{X_1}X_2\right).  
\end{split}
\end{align}

Let $x\in M$ %, $\Phi \in \Aut(\pi)$, 
and $y=\Psi(x)$.  
We claim $\lvert A\rvert^2(x) = \lvert A\rvert^2(y)$.  
Take horizontal vector fields $E_1, \dots, E_m$ on $M$ such that $E_1(x), \dots, E_m(x)$ is an orthonormal basis of $HM$ at $x$.  
Define horizontal vector fields $F_1 = \Psi_*E_1, \dots, F_m=\Psi_*E_m$.  
Since $F_1(y), \dots, F_m(y)$ is an orthonormal basis of $HM$ at $y$, \eqref{eq:AutAux} implies
%\begin{align*}
%2A_{F_i}F_j
%&=2A_{\Psi_*E_i}\Psi_*E_j
%=\mathcal{V}[\Psi_*E_i, \Psi_*E_j]
%=\mathcal{V}\left(\Psi_*[E_i, E_j]\right)
%=\Psi_*\left(\mathcal{V}[E_i, E_j]\right)\\
%&=2\Psi_*\left(A_{E_i}E_j\right), 
%\end{align*}
%\begin{align*}
%g(A_{F_i}F_j, A_{F_i}F_j)
%=g\left( \Psi_*(A_{E_i}E_j), \Psi_*(A_{E_i}E_j)\right)
%=g\left( A_{E_i}E_j, A_{E_i}E_j\right).  
%\end{align*}
%\[
%2A_{\Psi_*E_i}\Psi_*E_j
%=\mathcal{V}[\Psi_*E_i, \Psi_*E_j]
%=\mathcal{V}\left(\Psi_*[E_i, E_j]\right)
%=\Psi_*\left(\mathcal{V}[E_i, E_j]\right)
%=2\Psi_*\left(A_{E_i}E_j\right)
%\]
%\[
%g\left( A_{\Psi_*E_i}\Psi_*E_j, A_{\Psi_*E_i}\Psi_*E_j\right)
%=g\left( \Psi_*(A_{E_i}E_j), \Psi_*(A_{E_i}E_j)\right)
%=g\left( A_{E_i}E_j, A_{E_i}E_j\right)
%\]
%\begin{align*}
%\lvert A\rvert^2(x)&=\sum_{i, j=1}^m g(A_{E_i(x)}E_j(x), A_{E_i(x)}E_j(x))\\
%&=\sum_{i, j=1}^m g(A_{F_i(y)}F_j(y), A_{F_i(y)}F_j(y))=\lvert A\rvert^2(y).  
%\end{align*}
\begin{align*}
\lvert A\rvert^2(x)&=\sum_{i, j=1}^m g_x(A_{E_i}E_j, A_{E_i}E_j)
=\sum_{i, j=1}^m g_y(A_{F_i}F_j, A_{F_i}F_j)=\lvert A\rvert^2(y).  
\end{align*}
This proves $\Psi^*\lvert A\rvert=\lvert A\rvert$.  
\end{proof}

\subsection{Proof of Theorem \ref{Thm:irrep}}\label{sect:proofirrep}

Let $E\to B$ be a real vector bundle of rank $k+1\ge 1$, 
$\langle \cdot, \cdot \rangle$ a fiberwise inner product on $E$, 
and $UE=\{ s\in E \mid \langle s, s \rangle =1\}$ the unit sphere bundle of $(E, \langle \cdot, \cdot \rangle)$.  
With the restriction map $\pi: UE\to B$ of the projection $E\to B$, 
$UE$ is a $S^k$-bundle with the orthogonal group as structure group: 
\[
\Orth(k+1)\lact S^k \to UE \xrightarrow{\pi} B.  
\]
Let $\check{g}$ be a Riemannian metric on $B$ and $\nabla$ a metric connection on $(E, \langle \cdot, \cdot \rangle)$.   
We define the connection metric $g=\check{g}\oplus \overset{\circ}{g}_k$ on $UE$ 
with respect to the Ehresmann connection induced by $\nabla$.  
With the metric $g$, $\pi$ is a Riemannian submersion with totally geodesic fibers: 
\[
\Orth(k+1)\overset{\text{isom}}{\lact} S^k(1) \to (UE, g) \xrightarrow{\pi} (B, \check{g}).  
\]
Define the fiberwise symmetric bilinear form $\xi$ on the vector bundle $E$ by 
\begin{align}\label{eq:DefXi}
\xi(s_1, s_2)=\sum_{i, j=1}^m\langle R^{\nabla}(\check{E}_i, \check{E}_j)s_1, R^{\nabla}(\check{E}_i, \check{E}_j)s_2\rangle
\end{align}
for $s_1, s_2 \in E$.  
Here, $R^{\nabla}$ %: TB\times TB \times E \to E$ 
is the curvature of $\nabla$, 
%$m$ is the dimension of $B$, 
$m=\dim B$, 
and $\check{E}_1, \dots, \check{E}_m$ is an orthonormal basis for $(B, \check{g})$.  
$\xi$ does not depend on the choice of $\check{E}_1, \dots, \check{E}_m$.  
$\xi=0$ if and only if $\nabla$ is flat.  
The following formula appears in \cite[p. 279]{GSW}.  See also \cite{BL}, \cite[\S 9]{Bla} for proof.  

\begin{lemma}\label{lem:ScalSasaki}
The norm $\lvert A\rvert$ of O'Neill tensor for $\pi$ satisfies 
\begin{equation}
\lvert A\rvert^2(s)=\frac{1}{4}\xi(s, s)
\end{equation}
%The scalar curvature of the metric $g$ defined above satisfies 
%\begin{align}
%R_g (s)=\left(\pi^*R_{\check{g}}\right)(s)+k(k-1)-\frac{1}{4}\xi(s, s) 
%\end{align}
for every $s\in UE$.  
\end{lemma}

\begin{proof}[Proof of Throrem \ref{Thm:irrep}]
Let $(B, \check{g})=G/H$ be a Riemannian homogeneous space.  
%Each $a\in G$ determines $\psi_a \in \Isom(B, \check{g})$.  
%We denote by $\psi_a: B\to B$ the isometry determined by $a\in G$.  
For an irreducible orthogonal representation $\phi: H\to \SO(k+1)$ of $H$ on $\real^{k+1}$, 
define the vector bundle $E=G\times_{\phi}\real^{k+1}\to B$ 
with fiberwise inner product $\langle \cdot, \cdot \rangle$.  
%If $\phi$ is a trivial representation, then $UE=B\times S^k$ 
%so that we have only to set $a=0$ and define the product metrics $\check{g}\oplus r^2\overset{\circ}{g}_k$ for $r>0$.   
The action of $G$ on $B$ lifts to $E$, and $\langle \cdot, \cdot \rangle$ is invariant under this action.  
That is, 
\begin{equation}\label{eq:Auxil1}
\langle a_*s_1, a_*s_2\rangle=\langle s_1, s_2\rangle
\end{equation}
for $a\in G$.  
%\[
%\xymatrix{
%G\times \real^{k+1}\ar[d]\ar[r]&G\ar[d]\\
%E\ar[r]&B
%}\]
In particular, the lifted action of $G$ on $E$ preserves $UE$.  
%We write $\Psi_a$ for the diffeomorphism of $UE$ determined by $a\in G$.  

Let $\omega$ be a $G$-invariant principal $H$-connection on $G\to B$.  
The associated metric connection $\nabla$ on $E$ is invariant under the action of $G$.  %; 
%$a_*\nabla_{\check{X}}s=\nabla_{a_*\check{X}}a_*s$.  
The curvature $R^{\nabla}$ is thus invariant under $G$.  That is, 
\begin{equation}\label{eq:Auxil2}
a_*\left(R^{\nabla}(\check{X}, \check{Y})s\right)=R^{\nabla}(a_*\check{X}, a_*\check{Y})a_*s  
\end{equation}
for $a\in G$.  

Define $\xi$ by \eqref{eq:DefXi}.  
Let $V$ be the fiber of $E$ over the coset $o=eH$, where $e$ is the identity element of $G$.  
Denote by $\xi_o$ and $\langle \cdot, \cdot \rangle_o$ 
the restrictions of $\xi$ and $\langle \cdot, \cdot \rangle$ to $V$, respectively.  
$\xi_o$ is invariant under the restricted action of $H$ on $V$.  
Indeed, 
\begin{align*}
\xi_o(a_*s_1, a_*s_2)
&=\sum_{i, j=1}^m\langle R^{\nabla}(\check{E}_i, \check{E}_j)a_*s_1, R^{\nabla}(\check{E}_i, \check{E}_j)a_*s_2\rangle_o\\
&=\sum_{i, j=1}^m\langle a_* \left(R^{\nabla}(a^{-1}_*\check{E}_i, a^{-1}_*\check{E}_j)s_1\right), a_* \left(R^{\nabla}(a^{-1}_*\check{E}_i, a^{-1}_*\check{E}_j)s_2\right)\rangle_o &&\because\eqref{eq:Auxil2}\\
&=\sum_{i, j=1}^m\langle R^{\nabla}(a^{-1}_*\check{E}_i, a^{-1}_*\check{E}_j)s_1, R^{\nabla}(a^{-1}_*\check{E}_i, a^{-1}_*\check{E}_j)s_2\rangle_o &&\because\eqref{eq:Auxil1}\\
&=\sum_{i, j=1}^m\langle R^{\nabla}(\check{E}_i, \check{E}_j)s_1, R^{\nabla}(\check{E}_i, \check{E}_j)s_2\rangle_o
=\xi_o(s_1, s_2)
\end{align*}
for $a\in H$ since $\xi$ is independent of the orthonormal bases at $o$ chosen.  
%and $H$ acts on $B$ 
Since the representation $H\to \SO(V)$ is irreducible, 
%$\xi_o$ is invariant under this irreducible representation.  
it follows from Schur's lemma that $\xi_o$ and $\langle \cdot, \cdot \rangle_o$ are proportional.  
%Lemma \ref{lem:ScalSasaki} implies that $\lvert A\rvert$ 
%the norm of O'Neill tensor for $\pi: (UE, g)\to (B, \

%It follows that $\xi$ on $E$ defined by \eqref{eq:DefXi} is $H$-invariant.  
%We claim that $\xi$ is invariant under the action $G\lact E$.  
Let $g=\check{g}\oplus \overset{\circ}{g}_k $ be the connection metric on $UE$ 
with respect to the Ehresmann connection induced by $\nabla$.  
Lemma \ref{lem:ScalSasaki} implies that $\lvert A\rvert$, 
the norm of O'Neill tensor for $\pi: (UE, g)\to (B, \check{g})$, is constant on $\pi^{-1}(o)=UE \cap V$.  
On the other hand, it follows from Lemma \ref{lem:AutChar} that $G$ acts on $UE$ by automorphisms of the Riemannian submersion $(UE, g)\to (B, \check{g})$.  
Since $G$ acts on $UE$ fiber-transitively, we conclude from Lemma \ref{lem:orbitconst} that 
$\lvert A\rvert$ is constant on $UE$.  

Set $a:=\lvert A\rvert$.  %If the metric $\check{g}$ is flat, then $a=0$ since the principal connection $\omega$ is flat.  
%If $\phi$ is a trivial representation, then $UE=B\times S^k$ and $a=0$.  
The O'Neill formula \eqref{eq:ONeill'} implies that 
the connection metric 
$g(r) = \check{g} \oplus r^2 \overset{\circ}{g}_k$ for each $r>0$, with respect to the same Ehresmann connection as $g$, 
has constant scalar curvature $R(\check{g})+k(k-1)/r^2-a^2r^2$.  
%as asserted in Theorem \ref{Thm:irrep}.  
% satisfies all the properties in Theorem \ref{Thm:irrep}.  
%See Sect.\ \ref{sect:ConnMet} and \eqref{eq:ONeill'}.  
By construction, the projection 
$(UE, g(r))\to (B, \check{g})$ is a Riemannian submersion 
with totally geodesic fibers, and the typical fiber is isometric to the round sphere $S^k(r)$ (cf. Sect.\ \ref{sect:ConnMet}).  
\end{proof}

\subsection{Fiberwise join of sphere bundles}\label{sect:FiberJoin}

%\subsection{Join of spheres}\label{sect:join}
%\begin{remark}
%In this subsection, we always assume $k_1, k_2 \ge 1$.  
%\end{remark}

For $i=1, 2$, let $S^{k_i}$ be the sphere of dimension $k_i$ 
and $\Orth(k_i+1)\lact S^{k_i}$ be the transitive action ($k_1, k_2 \ge 0$).  
%In this subsection, we always assume $k_1, k_2 \ge 1$.  
We introduce the differential structure on their join 
\[
S^{k_1}*S^{k_2}=\frac{S^{k_1}\times [0, T]\times S^{k_2}}{0\times S^{k_1}, \ T\times S^{k_2}}
\]
through polar coordinates.  
%the sphere of dimension $k_1+k_2+1$, 
The action $\Orth(k_1+1)\times \Orth(k_2+1) \lact S^{k_1}\times [0, T]\times S^{k_2}$ 
descends smoothly to $S^{k_1}*S^{k_2}$, 
and $S^{k_1} * S^{k_2}$ is equivariantly diffeomorphic to $S^{k_1+k_2+1}$.  
Let 
\begin{equation}
\hat{\rho}: S^{k_1}\times (0, T)\times S^{k_2} \to S^{k_1}*S^{k_2}
\end{equation}
be the composition map 
$S^{k_1}\times (0, T) \times S^{k_2} \hookrightarrow S^{k_1} \times [0, T] \times S^{k_2} \twoheadrightarrow S^{k_1} * S^{k_2}$.  
$\hat{\rho}$ is an $\Orth(k_1+1)\times \Orth(k_2+1)$-equivariant injective 
local diffeomorphism whose image is dense in $S^{k_1}*S^{k_2}$.  
The following is well-known (cf.\ \cite[pp. 213--214]{KW}, \cite[4.6]{Bes1}, \cite[3.4--4.1]{Peter}).  
%
%The following is well known.  \cite[3.4--4.1]{Peter}).  
\begin{lemma}
Let $f_i: (0, T)\to \real$ be a strictly positive $C^{\infty}$ function for $i=1, 2$, 
and consider the doubly warped product metric 
$f_1^2(t) \overset{\circ}{g}_{k_1}+dt^2 +f_2^2(t) \overset{\circ}{g}_{k_2}$ on $S^{k_1}\times (0, T) \times S^{k_2}$.  

The scalar curvature of this metric %The scalar curvature of the doubly warped product metric 
%$f_1^2(t) \overset{\circ}{g}_{k_1}+dt^2 + f_2^2(t) \overset{\circ}{g}_{k_2}$ on $S^{k_1}\times (0, T)\times S^{k_2}$ 
is equal to 
\begin{align}\label{eq:ScalJoin}
-2k_1\frac{f_1''}{f_1}+k_1(k_1-1)\frac{1-(f_1')^2}{f_1^2}
-2k_2\frac{f_2''}{f_2}+k_2(k_2-1)\frac{1-(f_2')^2}{f_2^2}
-2k_1k_2\frac{f_1'f_2'}{f_1f_2}
\end{align}
on $S^{k_1} \times (0, T) \times S^{k_2}$.  

This metric %The metric $f_1^2(t) \overset{\circ}{g}_{k_1}+dt^2 +f_2^2(t) \overset{\circ}{g}_{k_2}$
%extends to an $\Orth(k_1+1)\times \Orth(k_2+1)$-invariant $C^{\infty}$ Riemannian metric on $S^{k_1}*S^{k_2}$ if and only if 
is the pullback by $\hat{\rho}$ of 
an $\Orth(k_1+1)\times \Orth(k_2+1)$-invariant $C^{\infty}$ Riemannian metric on $S^{k_1}*S^{k_2}$ 
if and only if $f_1$, $f_2$ extend to $C^{\infty}$ functions on $[0, T]$ and satisfy
\begin{equation}\begin{split}\label{eq:BC}
%f_1(0)=0, \ f_1'(0)=1, \ f_1^{(2l)}(0)=0, 
%\qquad f_1(T)>0, f_1^{(2l-1)}(T)=0, \\
%f_2(T)=0, \ f_2'(T)=-1, \ f_2^{(2l)}(T)=0, 
%\qquad f_2(0)>0, \ f_2^{(2l-1)}(0)=0
f_1(0)>0, \ f_1^{(2l-1)}(0)=0, 
\qquad &f_1(T)=0, \ f_1'(T)=-1, \ f_1^{(2l)}(T)=0, \\
f_2(T)>0, f_2^{(2l-1)}(T)=0, 
\qquad &f_2(0)=0, \ f_2'(0)=1, \ f_2^{(2l)}(0)=0
\end{split}\end{equation}
for every $l\ge 1$.  
%\end{lemma}
%%\begin{remark}
%%We implicitly assume $k_1, k_2 \ge 1$.  
%%\end{remark}
%
%\begin{lemma}
%Let $f_i: (0, T)\to \real$ be a strictly positive $C^{\infty}$ function for $i=1, 2$.  

\end{lemma}

%\begin{remark}
%In this subsection, we always assume $k_1, k_2 \ge 1$.  
%\end{remark}

Let $E_i \to B$ be a real vector bundle of rank $k_i+1\ge 1$, 
$\langle \cdot, \cdot \rangle_i$ a fiberwise inner product on $E_i$, 
and $\nabla^i$ a metric connection on $\left( E_i, \langle \cdot, \cdot \rangle_i \right)$ for $i=1, 2$.  
%
%In this subsection, we always assume $k_1, k_2 \ge 1$.  
%
For a Riemannian metric $\check{g}$ on $B$, 
define the connection metric $g_i=\check{g} \oplus \overset{\circ}{g}_{k_i}$ 
on the unit sphere bundle $\pi_i: UE_i\to B$ % of $E_i$ 
with respect to the Ehresmann connection induced by $\nabla^i$ 
as in Sect.\ \ref{sect:proofirrep}.  
We introduce the fiberwise join of $\pi_1: (UE, g_1) \to (B, \check{g})$ and $\pi_2: (UE, g_2) \to (B, \check{g})$ as follows.  

Take the direct sum connection $\nabla^1\oplus \nabla^2$ on $E_1\oplus E_2$ 
and consider the induced Ehresmann connection $HM$ on the unit sphere bundle $\pi: M=U(E_1\oplus E_2)\to B$.  
For an $\Orth(k_1+1)\times \Orth(k_2+1)$-invariant metric $\hat{g}$ on $S^{k_1+k_2+1}\cong S^{k_1}*S^{k_2}$, 
define the connection metric $g=\check{g}\oplus \hat{g}$ on $U(E_1\oplus E_2)$ with respect to $HM$.  
%
%Take an $\Orth(k_1+1)\times \Orth(k_2+1)$-invariant metric $\hat{g}$ on $S^{k_1}*S^{k_2}$ 
%and consider 
%\begin{enumerate}
%\item the orthonormal frame bundle $\Orth(k_i+1) \to P_i \to B$ with connection.  %$\omega_i$.  
%\item their fiberwise product $\Orth(k_1+1) \times \Orth(k_2+2) \to P \to B$ with connection $\omega$.  
%\item the associated bundle $\Orth(k_1+1) \times \Orth(k_2+2) \lact S^{k_1} * S^{k_2} \to M \xrightarrow{\pi} B$ 
%and connection 
%with respect to the action described above.  %in Sect. \ref{sect:join}.  
%\item the connection metric $g= \check{g} \oplus \hat{g}$ on $M$.  
%\end{enumerate}
%
The fiberwise join of $\pi_1$ and $\pi_2$ with respect to $\hat{g}$ is the Riemannian submersion $\pi : \left(U(E_1\oplus E_2), g\right)\to (B, \check{g})$ 
with totally geodesic fibers.  
%We write $\pi = \pi _1 \underset{B}{*} \pi_2$, 
%$M= UE_1 \underset{B}{*} UE_2$, and $g= g_1 \overset{\hat{g}}{\underset{B}{*}} g_2$.  
Summarizing the notations, 
\[
\Orth(k_1+1) \times \Orth(k_2+1) \overset{\text{isom}}{\lact} (S^{k_1} * S^{k_1}, \hat{g})
\to (U(E_1\oplus E_2), g)
\xrightarrow{\pi} (B, \check{g}).  
\]
%\[
%\Orth(k_1+1) \times \Orth(k_2+1) \overset{\text{isom}}{\lact} (S^{k_1} * S^{k_1}, \hat{g})
%\to \left(UE_1\underset{B}{*} UE_2, g_1 \overset{\hat{g}}{\underset{B}{*}} g_2\right)
%\xrightarrow{\pi_1\underset{B}{*}\pi_2} (B, \check{g}).  
%\]
%$M_1 \underset{B}{*} M_2$ is topologically the unit sphere bundle of the Whitney sum $E_1 \oplus E_2$.  

Equivalently, we can describe the fiberwise join in terms of principal bundles.  
Consider (1) the orthonormal frame bundle $P_i \to B$ of $(E_i, \langle \cdot, \cdot \rangle)$ and principal connection $\omega_i$ associated with $\nabla^i$, 
(2) their fiberwise product $P=P_1 \underset{B}{\times} P_2\to B$ and $\omega=\omega_1\underset{B}{\times}\omega_2$, 
and (3) the bundle $UE_1 \underset{B}{*} UE_2:=P\times_{\Orth(k_1+1)\times \Orth(k_2+1)} (S^{k_1}*S^{k_2})\to B$ 
and Ehresmann conection associated with $(P, \omega)$.  
%for the cohomogeneity-one action described above.  
The unit sphere bundle $\pi: U(E_1\oplus E_2)\to B$ with $HM$ 
is isomorphic to $UE_1 \underset{B}{*} UE_2\to B$ with this Ehresmann connection.  
Symbolically, 
\begin{equation*}
M=U(E_1\oplus E_2) \cong UE_1 \underset{B}{*} UE_2.  
\end{equation*}

%\begin{definition}%[Fiberwise join]
%For $i=1, 2$, let $\SO(k_i+1)\lact S^{k_i}(1)\to (M_i, g_i)\xrightarrow{\pi_i} (B, \check{g})$ be a Riemannian with totally geodesic %fibers 
%each of which is isometric to the round sphere.  
%Take an $\SO(k_1+1)\times \SO(k_2+1)$-invariant metric 
%$\hat{g}=dt^2 + f_1^2(t) \overset{\circ}{g}_{k_1}+f_2^2(t) \overset{\circ}{g}_{k_2}$ on $S^{k_1}*S^{k_2}$.  
%We define the fiberwise join...
%\[
%\SO(k_1+1)\times \SO(k_2+1)\lact \left(S^{k_1}*S^{k_2}, \hat{g}\right)
%\to \left(M_1\underset{B}{*}M_2, g(f_1, f_2) \right)\xrightarrow{\pi_1\underset{B}{*}\pi_2} (B, \check{g})
%\]
%\end{definition}

\begin{lemma}\label{lem:PropRhoJoin}
The equivariant map $\hat{\rho}: S^{k_1} \times (0, T)\times S^{k_2} \to S^{k_1}*S^{k_2}$ induces 
a map %morphism 
\[
\rho: (0, T) \times \left( UE_1 \underset{B}{\times} UE_2 \right) \to UE_1 \underset{B}{*} UE_2.  
\]
%of bundles over $B$ covering the identity.  
$\rho$ is an injective local diffeomorphism whose image is dense in $UE_1 \underset{B}{*} UE_2$.  
%$\rho$ is an isometric embedding with dense image if 
%we endow $(0, T) \times \left( UE_1 \underset{B}{\times} UE_2 \right)$ with the connection metric 
%\[
%\check{g} \oplus \left( f_1^2(t) \overset{\circ}{g}_{k_1}+dt^2 +f_2^2(t) \overset{\circ}{g}_{k_2}\right)
%\]
\end{lemma}
\begin{proof}
%The map $\id \times \hat{\rho}: P\times S^{k_1} \times (0, T)\times S^{k_2}\to P\times (S^{k_1}*S^{k_2})$ descends to $\rho$.  
%Here, $UE_1 \underset{B}{\times} UE_2=P\times_{\Orth(k_1+1)\times \Orth(k_2+1)} (S^{k_1}\times S^{k_2}) $ is the fiberwise product.  
Define $\rho$ so that the diagram
\begin{align*}
%\xymatrix{
%P \times (-T, T) \times S^k \ar[d]\ar[r]& (-T, T) \times UE \ar[d]\\
%P \ar[r]& B
%}
\xymatrix{
P\times S^{k_1} \times (0, T)\times S^{k_2}\ar[r]^-{\id \times \hat{\rho}}\ar[d]&P\times (S^{k_1}*S^{k_2})\ar[d]\\
(0, T)\times \left(UE_1 \underset{B}{\times} UE_2\right) \ar[r]^-{\rho}& UE_1 \underset{B}{*} UE_2  
}
%\xymatrix{
%P \times (-T, T) \times S^k \ar[d]\ar[r]& (-T, T) \times UE \ar[d]\\
%P \times S(S^k) \ar[d]\ar[r]& \widetilde{UE} \ar[d]\\
%P \ar[r]& B
%}
\end{align*}
commutes.  
Here, $UE_1 \underset{B}{\times} UE_2=P\times_{\Orth(k_1+1)\times \Orth(k_2+1)} (S^{k_1}\times S^{k_2}) $ is the fiberwise product.  
$\rho$ is well-defined since $\hat{\rho}$ is $\Orth(k_1+1)\times \Orth(k_2+1)$-equivariant.  
Here, the vertical arrows are quotient maps with respect to the diagonal actions of $\Orth(k_1+1)\times \Orth(k_2+1)$.  
Note that the associated bundle $P\times_{\Orth(k_1+1)\times \Orth(k_2+1)}\left( S^{k_1} \times (0, T)\times S^{k_2}\right)$ 
is canonically isomorphic to $(0, T)\times \left(UE_1 \underset{B}{\times} UE_2\right)$, 
for $\Orth(k_1+1)\times \Orth(k_2+1)$ acts on $(0, T)$ trivially.  
Since $\hat{\rho}$ is an injective local diffeomorphism with dense image, so is $\rho$.  
\end{proof}
%
%\begin{proof}
%Let $\Orth(k_1+1) \times \Orth(k_2+2) \to P \to B$ be 
%the fiberwise product of the orthonormal frame bundles of $(E_i, \langle \cdot, \cdot \rangle_i)$.  
%\[
%a
%\]
%
%Since the submanifold $(0, T) \times S^{k_1} \times S^{k_2}$ of $S^{k_1}*S^{k_2}$ is invariant under the 
%action of $\Orth(k_1+1) \times \Orth(k_2+1)$, 
%
%$\Orth(k_1+1) \times \Orth(k_2+1)$ acts on $(0, T) \times S^{k_1} \times S^{k_2}$ by 
%\[
%(g_1, g_2). (t, \hat{x}_1, \hat{x}_2)=(t, g_1.\hat{x}_1, g_2.\hat{x}_2).  
%\]
%Since this action is trivial on the $(0, T)$-factor, 
%
%\end{proof}

\begin{lemma}\label{lem:Joinf1f2}
Let $\hat{g}$ be an $\Orth(k_1+1)\times \Orth(k_2+1)$-invariant metric on $S^{k_1}*S^{k_2}$ 
such that $\hat{\rho}^*\hat{g}=f_1^2(t) \overset{\circ}{g}_{k_1}+dt^2 +f_2^2(t) \overset{\circ}{g}_{k_2}$ 
on $S^{k_1}\times (0, T) \times S^{k_2}$.  
Consider the 
fiberwise join $\pi : \left(UE_1 \underset{B}{*} UE_2, g\right)\to (B, \check{g})$ with respect to $\hat{g}$.  
%metric of fiberwise join metric on $UE_1 \underset{B}{*} UE_2$ with respect to $\hat{g}$.  

Then, the pullback metric $\rho^*g$ on $(0, T) \times \left( UE_1 \underset{B}{\times} UE_2 \right)$ is equal to the connection metric 
$\check{g}\oplus \left(f_1^2(t) \overset{\circ}{g}_{k_1}+dt^2 +f_2^2(t) \overset{\circ}{g}_{k_2}\right)$ 
with respect to the Ehresmann associated with $\omega$.    
%with respect to the Ehresmann connection %associated with $\omega$.  
%
Consequently, the scalar curvature of $\rho^*g$ satisfies
%On the open dense subbundle $(0, T)\times \left(UE_1\underset{B}{\times} UE_2\right)$, 
\begin{equation}\label{eq:ODE}
\begin{split}
%R(\rho^*g)&=
R(\rho^*g)-(\pi\circ\rho)^*R(\check{g})&=
%\rho^*(R(g))=
%(\pi\circ\rho)^*(R(\check{g}))-2k_1k_2\frac{f_1'f_2'}{f_1f_2}\\
-2k_1\frac{f_1''}{f_1}+k_1(k_1-1)\frac{1-(f_1')^2}{f_1^2}-\lvert A^1\rvert^2f_1^2\\
&\quad-2k_2\frac{f_2''}{f_2}+k_2(k_2-1)\frac{1-(f_2')^2}{f_2^2}-\lvert A^2\rvert^2f_2^2
-2k_1k_2\frac{f_1'f_2'}{f_1f_2}
\end{split}
\end{equation}
on $(0, T) \times \left( UE_1 \underset{B}{\times} UE_2 \right)$.  
Here, $\lvert A^i\rvert$ is the norm of O'Neill tensor 
for $\pi_i$.  %: (UE_i, \check{g}\oplus \overset{\circ}{g}_{k_i}) \to (B, \check{g})$.  
\end{lemma}

\begin{proof}
%To show 
$\rho^*\tilde{g}=\check{g}\oplus \left( f_1^2(t) \overset{\circ}{g}_{k_1}+dt^2 +f_2^2(t) \overset{\circ}{g}_{k_2} \right)$ 
%it suffices to check that 
since 
the map $\rho$ preserves the horizontal and vertical bundles 
as well as the lengths of horizontal and vertical vectors, respectively.  
%We omit the details.  

We prove \eqref{eq:ODE}.  
Let $\lvert A\rvert$ be the norm of O'Neill tensor for the Riemannian submersion $\pi\circ\rho: (0, T) \times \left( UE_1 \underset{B}{\times} UE_2 \right)\to B$.  
%$\tilde{\pi}: (\widetilde{UE}, \tilde{g})\to (B, \check{g})$.  
In view of the O'Neill formula \eqref{eq:ONeill} for scalar curvature and \eqref{eq:ScalJoin}, 
we have only to show 
\begin{equation}\label{eq:ONeillJoin}
\lvert A\rvert^2=\lvert A^1\rvert^2 f_1^2+\lvert A^2\rvert^2 f_2^2.  
\end{equation}
%Let $\check{X}, \check{Y}$ be vector fields on $B$ and $X, Y$ their horizontal lifts to $UE$.  
%The horizontal lifts $\tilde{X}$, $\tilde{Y}$ of $\check{X}, \check{Y}$ to $UE\times (-T, T)$ agree with the trivial extensions of $X, Y$ to $UE\times (-T, T)$.  
%It follows $\tilde{A}_{\tilde{X}}\tilde{Y}=\tilde{\mathcal{V}}[\tilde{X}, \tilde{Y}]=\mathcal{V}[X, Y]=A_XY$.  
%In particular, $\tilde{A}_{\tilde{X}}\tilde{Y}$ is tangent to the fibers of $\pi: UE\to B$.  
%Therefore, 
%
%Let $\lvert A\rvert$ be the norm of O'Neill tensor for the Riemannian submersion $(0, T) \times \left( UE_1 \underset{B}{\times} UE_2 \right)\to B$.  
%As in the proof of Lemma \ref{lem:PullbackMet}, 
Take basic vector fields $X$, $Y$ on $(0, T) \times \left( UE_1 \underset{B}{\times} UE_2 \right)$.  
Since $A_XY=A^1_XY+A^2_XY$, $A^1_XY\perp A^2_XY$, and $A^i_XY$ is tangent to the fibers of $\pi_i: UE_i\to B$, 
there holds
\begin{align*}
%\lvert A_XY\rvert^2
(\rho^*g)(A_XY, A_XY)
&=\left(f_1^2 \overset{\circ}{g}_{k_1}+dt^2 +f_2^2 \overset{\circ}{g}_{k_2}\right)(A_XY, A_XY)\\
&=f_1^2\overset{\circ}{g}_{k_1}\left(A^1_XY, A^1_XY\right)+f_2^2\overset{\circ}{g}_{k_2}\left(A^2_XY, A^2_XY\right).  
%&=\lvert (A_1)_XY\rvert^2 f_1^2+\lvert (A_2)_XY\rvert^2 f_2^2.  
\end{align*}
%Hence, taking $E_1, \dots, E_m$ to be the horizontal lifts  
%Hence $\lvert A\rvert^2 = \lvert A_1\rvert^2f_1^2+\lvert A_2\rvert^2f_2^2$.  
Hence, taking an orthonormal vector field $E_1, \dots, E_m$ to be basic, we obtain
\begin{align*}
\lvert A\rvert^2
&=\sum_{i, j=1}^m (\rho^*g)(A_{E_i}E_j, A_{E_i}E_j)\\
&=\sum_{i, j=1}^m f_1^2\overset{\circ}{g}_{k_1}\left(A^1_{E_i}E_j, A^1_{E_i}E_j\right)
+\sum_{i, j=1}^m f_2^2\overset{\circ}{g}_{k_2}\left(A^2_{E_i}E_j, A^2_{E_i}E_j\right)\\
&=\lvert A^1\rvert^2 f_1^2+\lvert A^2\rvert^2 f_2^2.  
\end{align*}
This proves \eqref{eq:ONeillJoin}.  %See \eqref{eq:DefONeill}.  
%\eqref{eq:ONeill} and \eqref{eq:ScalJoin} then imply \eqref{eq:ODE}.  
\end{proof}

Let $\cn_{\mathbf{k}}(t)$, $\sn_{\mathbf{k}}(t)$ be the Jacobi elliptic functions (\cite{WW}, \cite{Weis}) 
and $4K(\mathbf{k})$ their fundamental period ($0\le \mathbf{k}<1$).  
$\cn_{\mathbf{k}}, \sn_{\mathbf{k}}: \real \to \real$ %is periodic with fundamental period $4K(\mathbf{k})$, $\cn_{\mathbf{k}}>0$ on $(-K(\mathbf{k}), K(\mathbf{k}))$, 
satisfy $\cn_{\mathbf{k}}^2+\sn_{\mathbf{k}}^2=1$, %$\cn_{\mathbf{k}}(-t)=\cn_{\mathbf{k}}(t)$, 
\begin{align*}
%\left( \frac{d\cn_k}{dt}\right)^2
(\cn_{\mathbf{k}}')^2
=(1-\cn_{\mathbf{k}}^2)(1-\mathbf{k}^2+\mathbf{k}^2\cn_{\mathbf{k}}^2), &&
%\frac{d^2\cn_k}{dt^2}
\cn_{\mathbf{k}}''
=-2\mathbf{k}^2\cn_{\mathbf{k}}^3-(1-2\mathbf{k}^2)\cn_{\mathbf{k}}, \\
(\sn_{\mathbf{k}}')^2=(1-\sn_{\mathbf{k}}^2)(1-\mathbf{k}^2\sn_{\mathbf{k}}^2), 
&&\sn_{\mathbf{k}}''=2\mathbf{k}^2\sn_{\mathbf{k}}^3-(1+\mathbf{k}^2)\sn_{\mathbf{k}}
\end{align*}
on $\real$.  
%We denote by $\bar{k}=\sqrt{1-k^2}$ the complementary modulus.  
$\cn_{\mathbf{k}} /\sqrt{1-\mathbf{k}^2}$, $\sn_{\mathbf{k}}$ are strictly positive on $(0, T)$ and satisfy 
the boundary conditions \eqref{eq:BC} for $f_1$, $f_2$ with $T=K(\mathbf{k})$.  

%The Jacobi elliptic function $\sn_{\mathbf{k}}: \real\to\real$ satisfies $\cn_{\mathbf{k}}^2+\sn_{\mathbf{k}}^2=1$, 
%\begin{align*}
%(\sn_{\mathbf{k}}')^2=(1-\sn_{\mathbf{k}}^2)(1-\mathbf{k}^2\sn_{\mathbf{k}}^2), 
%&&\sn_{\mathbf{k}}''=2\mathbf{k}^2\sn_{\mathbf{k}}^3-(1+\mathbf{k}^2)\sn_{\mathbf{k}}
%\end{align*}
%on $\real$.  
%$\sn_{\mathbf{k}}$ is strictly positive on $(0, T)$ and satisfies the boundary conditions \eqref{eq:BC} for $f_2$ with $T=K(\mathbf{k})$.  

\begin{lemma}\label{lem:ODEanJoin}
Assume $\check{R}$, $a_i\ge 0$, and $k_i\ge 1$ are constant for $i=1, 2$.  
Suppose $\gamma\in(0, \infty)$ and $\mathbf{k}\in [0, 1)$ satisfy
\begin{equation}\label{eq:QEAl}\begin{split}
(k_1+k_2)(k_1+k_2+3)\gamma^4 \mathbf{k}^2(1-\mathbf{k}^2)=a_1^2-(1-\mathbf{k}^2)a_2^2.  %, \\
%R-\check{R}=-2(k_1+k_2)(k_1+1)\gamma^2\mathbf{k}^2+(k_1+k_2)(k_1+k_2+1)\gamma^2-a_2^2/\gamma^2
\end{split}\end{equation}
%for a real number $R$.  
Then, $f_1(t)=\cn_{\mathbf{k}}(\gamma t)/(\gamma\sqrt{1-\mathbf{k}^2})$ and $f_2(t)=\sn_{\mathbf{k}}(\gamma t)/\gamma$ solve
\begin{equation}\label{eq:ODEanalyticJoin}
\begin{split}
R-\check{R}&=
-2k_1\frac{f_1''}{f_1}+k_1(k_1-1)\frac{1-(f_1')^2}{f_1^2}-a_1^2f_1^2\\
&\quad-2k_2\frac{f_2''}{f_2}+k_2(k_2-1)\frac{1-(f_2')^2}{f_2^2}-a_2^2f_2^2
-2k_1k_2\frac{f_1'f_2'}{f_1f_2}
\end{split}
\end{equation}
on $(0, T)$, are strictly positive on $(0, T)$, and satisfy the boundary conditions \eqref{eq:BC} with $T=K(\mathbf{k})/\gamma$, 
\begin{equation}\label{eq:Rdef}
R-\check{R}=-2(k_1+k_2)(k_1+1)\gamma^2\mathbf{k}^2+(k_1+k_2)(k_1+k_2+1)\gamma^2-a_2^2/\gamma^2.  
\end{equation}
\end{lemma}

\begin{proof}
For arbitrary $\gamma\in(0, \infty)$ and $\mathbf{k}\in[0, 1)$,  
$f_1(t)=\cn_{\mathbf{k}}(\gamma t)/(\gamma\sqrt{1-\mathbf{k}^2})$ and $f_2(t)=\sn_{\mathbf{k}}(\gamma t)/\gamma$ 
are strictly positive on $(0, T)$ and satisfy the boundary conditions \eqref{eq:BC} with $T=K(\mathbf{k})/\gamma$.  
Since 
\begin{align*}
(f_1'(t))^2&=\left(\frac{\cn_{\mathbf{k}}'(\gamma t)}{\sqrt{1-\mathbf{k}^2}}\right)^2
=-\frac{\mathbf{k}^2}{1-\mathbf{k}^2}\cn_{\mathbf{k}}^4(\gamma t)-\frac{1-2\mathbf{k}^2}{1-\mathbf{k}^2}\cn_{\mathbf{k}}^2(\gamma t)+1\\
&=-\gamma^4 \mathbf{k}^2(1-\mathbf{k})^2f_1^4(t)-\gamma^2(1-2\mathbf{k}^2)f_1^2(t)+1, \\
f_1''(t)%&=\gamma\frac{\cn_k''(\gamma t)}{\bar{k}}
%=-\gamma \frac{2k^2}{\bar{k}}\cn_k^3(\gamma t)-\gamma\frac{1-2k^2}{\bar{k}}\cn_k(\gamma t)\\
&=-2\gamma^4\mathbf{k}^2(1-\mathbf{k})^2f_1^3(t)-\gamma^2(1-2\mathbf{k}^2)f_1(t), 
\end{align*}
\begin{align*}
-\frac{f_1''}{f_1}=2\gamma^4\mathbf{k}^2(1-\mathbf{k})^2f_1^2+\gamma^2(1-2\mathbf{k}^2), &&
\frac{1-(f_1')^2}{f_1^2}=\gamma^4\mathbf{k}^2(1-\mathbf{k})^2f_1^2+\gamma^2(1-2\mathbf{k}^2), 
\end{align*}
we obtain
\begin{align}\label{eq:Aux02}
-2k_1\frac{f_1''}{f_1}+k_1(k_1-1)\frac{1-(f_1')^2}{f_1^2}
=k_1(k_1+3)\gamma^4 \mathbf{k}^2(1-\mathbf{k}^2)f_1^2+k_1(k_1+1)\gamma^2(1-2\mathbf{k}^2).  
\end{align}
Similarly, since 
\begin{align*}
(f_2'(t))^2&=(\sn'_{\mathbf{k}}(\gamma t))^2
=\mathbf{k}^2\sn_{\mathbf{k}}^4(\gamma t)-(1+\mathbf{k}^2)\sn_{\mathbf{k}}^2(\gamma t)+1\\
&=\gamma^4\mathbf{k}^2f_2^4(t)-\gamma^2(1+\mathbf{k}^2)f_2^2(t)+1, \\
%=(1-\sn_{\mathbf{k}}^2(\gamma t))(1-\mathbf{k}^2\sn_{\mathbf{k}}^2(\gamma t))\\
f_2''(t)&=2\gamma^4\mathbf{k}^2f_2^3(t)-\gamma^2(1+\mathbf{k}^2)f_2(t), 
\end{align*}
\begin{align*}
-\frac{f_2''}{f_2}=-2\gamma^4\mathbf{k}^2f_2^2+\gamma^2(1+\mathbf{k}^2), &&
\frac{1-(f_2')^2}{f_2^2}=-\gamma^4\mathbf{k}^2f_2^2+\gamma^2(1+\mathbf{k}^2), 
\end{align*}
we obtain
\begin{align}\label{eq:Aux01}
-2k_2\frac{f_2''}{f_2}+k_2(k_2-1)\frac{1-(f_2')^2}{f_2^2}%-a_2^2f_2^2
=-k_2(k_2+3)\gamma^4\mathbf{k}^2f_2^2+k_2(k_2+1)\gamma^2(1+\mathbf{k}^2).  
\end{align}
%\[(f'(t))^2=-\gamma^4 \mathbf{k}^2(1-\mathbf{k})^2f^4(t)-\gamma^2(1-2\mathbf{k}^2)f^2(t)+1, \]
%We have already seen 
%in \eqref{eq:Scalf1}.  
The relation $\cn_{\mathbf{k}}^2+\sn_{\mathbf{k}}^2=1$ yields
\begin{align}\label{align:conv}
(1-\mathbf{k}^2)f_1^2+f_2^2=1/\gamma^2, &&
(1-\mathbf{k}^2)f_1f_1'+f_2f_2'=0, 
\end{align}
\begin{align*}
-2k_1k_2\frac{f_1'f_2'}{f_1f_2}
&=2k_1k_2\frac{f_2'}{f_1^2f_2}\frac{f_2f_2'}{1-\mathbf{k}^2}
=\frac{2k_1k_2}{1-\mathbf{k}^2}\frac{(f_2')^2}{f_1^2}.  
\end{align*}
Since we can also write
\begin{align*}
(f_2'(t))^2&=(\sn'_{\mathbf{k}}(\gamma t))^2
=(1-\sn_{\mathbf{k}}^2(\gamma t))(1-\mathbf{k}^2\sn_{\mathbf{k}}^2(\gamma t))\\
&=\gamma^2(1-\mathbf{k}^2)f_1^2(t)(1-\gamma^2\mathbf{k}^2f_2^2(t)), 
\end{align*}
there holds
\begin{align}\label{eq:Aux03}
-2k_1k_2\frac{f_1'f_2'}{f_1f_2}
%=2k_1k_2\gamma^2(1-\gamma^2\mathbf{k}^2f_2^2)
=-2k_1k_2\gamma^4\mathbf{k}^2f_2^2+2k_1k_2\gamma^2.  
\end{align}
Hence \eqref{eq:Aux01}, \eqref{eq:Aux02}, and \eqref{eq:Aux03} imply
\begin{align*}
&-2k_1\frac{f_1''}{f_1}+k_1(k_1-1)\frac{1-(f_1')^2}{f_1^2}
-2k_2\frac{f_2''}{f_2}+k_2(k_2-1)\frac{1-(f_2')^2}{f_2^2}
-2k_1k_2\frac{f_1'f_2'}{f_1f_2}\\
&=k_1(k_1+3)\gamma^4 \mathbf{k}^2(1-\mathbf{k}^2)f_1^2+k_1(k_1+1)\gamma^2(1-2\mathbf{k}^2)\\
&\quad -k_2(k_2+3)\gamma^4\mathbf{k}^2f_2^2 +k_2(k_2+1)\gamma^2(1+\mathbf{k}^2)\\
&\quad -2k_1k_2\gamma^4\mathbf{k}^2f_2^2  +2k_1k_2\gamma^2\\
&=k_1(k_1+3)\gamma^4 \mathbf{k}^2(1-\mathbf{k}^2)f_1^2-k_2(2k_1+k_2+3)\gamma^4\mathbf{k}^2f_2^2\\
&\quad+(k_1+k_2)(k_1+k_2+1)\gamma^2-(2k_1^2-k_2^2+2k_1-k_2)\gamma^2\mathbf{k}^2.  
\end{align*}
Under the assumption \eqref{eq:QEAl}, 
\[
1-\mathbf{k}^2 \ : \ 1
\quad =\quad k_1(k_1+3)\gamma^4 \mathbf{k}^2(1-\mathbf{k}^2)-a_1^2\ :\ -\left(k_2(2k_1+k_2+3)\gamma^4\mathbf{k}^2+a_2^2\right)
\]
holds so that 
%\[
%(k_1+k_2)(k_1+k_2+3)\gamma^4 \mathbf{k}^2(1-\mathbf{k})^2=a_1^2-(1-\mathbf{k}^2)a_2^2
%\]
\begin{align*}
&-2k_1\frac{f_1''}{f_1}+k_1(k_1-1)\frac{1-(f_1')^2}{f_1^2}-a_1^2f_1^2\\
&-2k_2\frac{f_2''}{f_2}+k_2(k_2-1)\frac{1-(f_2')^2}{f_2^2}-a_2^2f_2^2-2k_1k_2\frac{f_1'f_2'}{f_1f_2}\\
&=\left(k_1(k_1+3)\gamma^4 \mathbf{k}^2(1-\mathbf{k}^2)-a_1^2\right)f_1^2-\left(k_2(2k_1+k_2+3)\gamma^4\mathbf{k}^2+a_2^2\right)f_2^2\\
&\quad+(k_1+k_2)(k_1+k_2+1)\gamma^2-(2k_1^2-k_2^2+2k_1-k_2)\gamma^2\mathbf{k}^2\\
&=-\left(k_2(2k_1+k_2+3)\gamma^4\mathbf{k}^2+a_2^2\right)/\gamma^2\\
&\quad+(k_1+k_2)(k_1+k_2+1)\gamma^2-(2k_1^2-k_2^2+2k_1-k_2)\gamma^2\mathbf{k}^2\\
&=-(2k_1k_2+2k_1^2+2k_2+2k_1)\gamma^2\mathbf{k}^2+(k_1+k_2)(k_1+k_2+1)\gamma^2-a_2^2/\gamma^2\\
&=-2(k_1+k_2)(k_1+1)\gamma^2\mathbf{k}^2+(k_1+k_2)(k_1+k_2+1)\gamma^2-a_2^2/\gamma^2\\
&=R-\check{R}.  
\end{align*}
This proves \eqref{eq:ODEanalyticJoin}.  
\end{proof}

\begin{proof}[Proof of Theorem \ref{Thm:FiberJoin}]
%If $k_1=k_2=0$, then $U(E_1\oplus E_2)$ is a circle bundle with integrable connection.  
%In this case, the locally product metric $\check{g}\oplus r^2\overset{\circ}{g}_1$ has constant scalar curvature $R(\check{g})$ for every $r>0$.  
%If $k_1=0$ or $k_2=0$, 

%Since $R(\check{g})$, $\lvert A_1\rvert$, and $\lvert A_2\rvert$ are constant, 
%we have only to solve the ODE \eqref{eq:ODE} under the boundary conditions and positivity assumption.  
Note that the norm $\lvert A^i\rvert$ of O'Neill tensor for $\pi_i: (UE_i, g_i)\to (B, \check{g})$ is constant since 
$g_i$ and $\check{g}$ have constant scalar curvature (cf.\ Proposition \ref{cscRiemSubm}).  
$\lvert A^i\rvert=0$ if and only if $\nabla^i$ is flat.  
%When $k_1=0$ or $k_2=0$, then 

Let $\check{R}=R(\check{g})$, $a_i=\lvert A^i\rvert$.  
Suppose $\gamma\in (0, \infty)$ and $\mathbf{k}\in [0, 1)$ satisfy \eqref{eq:QEAl}, 
and set $R\in \real$ by \eqref{eq:Rdef}.  
%
%Let $R>0$ be a large  number.    
% every $R\in I$, t
%The simultaneous equations \eqref{eq:QEAl} with $\check{R}=R(\check{g})$ and $a_i=\lvert A_i\rvert$ 
%have solutions $\gamma\in (0, \infty)$, $\mathbf{k}\in [0, 1)$.  
%For such $\gamma$ and $\mathbf{k}$, 
Define $T=K(\mathbf{k})/\gamma$, $f_1(t)=\cn_{\mathbf{k}}(\gamma t)/(\gamma\sqrt{1-\mathbf{k}^2})$, and $f_2(t)=\sn_{\mathbf{k}}(\gamma t)/\gamma$.  
Since $f_1$ and $f_2$ are positive and satisfy the boundary conditions (Lemma \ref{lem:ODEanJoin}), 
there exists an $\Orth(k_1+1)\times \Orth(k_2+1)$-invariant $C^{\infty}$ Riemannian metric $\hat{g}$ on $S^{k_1}*S^{k_2}$ 
such that $\hat{\rho}^*\hat{g}=f_1^2(t) \overset{\circ}{g}_{k_1}+dt^2 +f_2^2(t) \overset{\circ}{g}_{k_2}$ 
on $S^{k_1}\times (0, T)\times S^{k_2}$ (Lemma \ref{lem:Joinf1f2}).  

Let $\pi: \left(U(E_1\oplus E_2), g\right) \to (B, \check{g})$ be the fiberwise join of $\pi_i : (UE, g_i) \to (B, \check{g})$ ($i=1, 2$).  
%
%Let $g=\check{g}\oplus\hat{g}$ be the fiberwise join metric on $U(E_1\oplus E_2)$.  
Since $f_1$ and $f_2$ solve \eqref{eq:ODEanalyticJoin}, it follows from \eqref{eq:ODE} 
that the scalar curvature of $\rho^*g$ is identically equal to $R$ on $(0, T)\times \left(UE_1\underset{B}{\times}UE_2\right)$.  
Therefore, $g$ has constant scalar curvature $R$ on $U(E_1\oplus E_2)$ 
since $\rho: (0, T)\times \left(UE_1\underset{B}{\times}UE_2\right)\to U(E_1\oplus E_2)$ has dense image (Lemma \ref{lem:PropRhoJoin}).  

For the rest of proof, we solve \eqref{eq:QEAl}, \eqref{eq:Rdef} explicitly.  
%We see in particular that there are two one-parameter families of constant scalar curvature metrics if $\nabla^1$ and $\nabla^2$ are not flat, 
%and that there is only one family otherwise.  
%For the rest of proof, we solve \eqref{eq:QEAl} and derive an explicit formula for the scalar curvature $R$ using \eqref{eq:Rdef}.  
Firstly, observe the following: 
If $\gamma\in(0, \infty)$, $\mathbf{k}\in (0, 1)$ satisfy \eqref{eq:QEAl} and if $a_1^2-(1-\mathbf{k}^2)a_2^2>0$, then %\eqref{eq:Rdef} implies
\begin{align}
\gamma^2&=\sqrt{\frac{1}{(k_1+k_2)(k_1+k_2+3)}}\sqrt{\frac{a_1^2-(1-\mathbf{k}^2)a_2^2}{\mathbf{k}^2(1-\mathbf{k}^2)}}>0, \label{align:gam}\\
R-\check{R}&=(k_1+k_2)\left(k_1+k_2+1-2(k_1+1)\mathbf{k}^2\right)\gamma^2-a_2^2/\gamma^2\label{align:R1}\\
\begin{split}&=\sqrt{\frac{k_1+k_2}{k_1+k_2+3}}\sqrt{\frac{a_1^2-(1-\mathbf{k}^2)a_2^2}{\mathbf{k}^2(1-\mathbf{k}^2)}}\left(k_1+k_2+1-2(k_1+1)\mathbf{k}^2\right)\\
&\quad - a_2^2\sqrt{(k_1+k_2)(k_1+k_2+3)}\sqrt{\frac{\mathbf{k}^2(1-\mathbf{k}^2)}{a_1^2-(1-\mathbf{k}^2)a_2^2}}.  \end{split}\label{align:R2}
%&=-2(k_1+k_2)(k_1+1)\gamma^2\mathbf{k}^2+(k_1+k_2)(k_1+k_2+1)\gamma^2-a_2^2/\gamma^2\\
%&=-2(k_1+1)\sqrt{\frac{k_1+k_2}{k_1+k_2+3}}\sqrt{\frac{\mathbf{k}^2\left(a_1^2-(1-\mathbf{k}^2)a_2^2\right)}{1-\mathbf{k}^2}}\\
%&\quad +(k_1+k_2+1)\sqrt{\frac{k_1+k_2}{k_1+k_2+3}}\sqrt{\frac{a_1^2-(1-\mathbf{k}^2)a_2^2}{\mathbf{k}^2(1-\mathbf{k}^2)}}\\
%&\quad - a_2^2\sqrt{(k_1+k_2)(k_1+k_2+3)}\sqrt{\frac{\mathbf{k}^2(1-\mathbf{k}^2)}{a_1^2-(1-\mathbf{k}^2)a_2^2}}
\end{align}
If $a_1=a_2=0$, then we take $\mathbf{k}=0$ and an arbitrary $\gamma>0$ so that $R-\check{R}=(k_1+k_2)(k_1+k_2+1)\gamma^2$.  
The resulting metrics are locally the direct products of $(B, \check{g})$ and the $(k_1+k_2+1)$-dimensional round spheres 
%$(S^{k_1+k_2+1}, \overset{\circ}{g}_{k_1+k_2+1})$ 
in this case.  
Interchanging $E_1$ and $E_2$ if necessary, 
assume $a_1> 0$.  %until the end of proof.  %in what follows.  %without loss of generality.  
\begin{itemize}
\item Case 1: $a_1>a_2$.  
For each $\mathbf{k}\in (0, 1)$, define $\gamma>0$ by \eqref{align:gam}.  
As $\mathbf{k}\to 0$, $\gamma\to \infty$.  
Hence $R\to \infty$ as $\mathbf{k}\to 0$ by \eqref{align:R1}.  
\item Case 2: $a_1=a_2>0$.  
We can take $\mathbf{k}=0$ and an arbitrary $\gamma >0$ so that $R-\check{R}=(k_1+k_2)(k_1+k_2+1)\gamma^2-a_2^2/\gamma^2$.  
We can also define $\gamma>0$ by \eqref{align:gam} for each $\mathbf{k}\in (0, 1)$.  
As $\mathbf{k}\to 0$, $\gamma\to\infty$.  
Hence $R\to \infty$ as $\mathbf{k}\to 0$ by \eqref{align:R1}.  
\item Case 3: $0<a_1<a_2$.  
For each $\mathbf{k}\in (1-a_1^2/a_2^2, 1)$, define $\gamma>0$ by \eqref{align:gam}.  
As $\mathbf{k}\to 1-a_1^2/a_2^2$, $\gamma\to 0$.  
Hence $R\to -\infty$ as $\mathbf{k}\to 1-a_1^2/a_2^2$ by \eqref{align:R1}.  
\end{itemize}
In all these cases, we can take $\mathbf{k}$ arbitrarily close to $1$.  Lastly, we see the limiting behavior of $R$ as $\mathbf{k}\to 1$.  
We claim: 
\begin{align}
&\text{If $k_2> k_1+1$, then $R\to\infty$ as $\mathbf{k}\to 1$.  }\label{align:limRpos}\\
&\text{If $k_2= k_1+1$, then $R\to 0$ as $\mathbf{k}\to 1$.  }\label{align:limRzero}\\
&\text{If $k_2< k_1+1$, then $R\to -\infty$ as $\mathbf{k}\to 1$.   }\label{align:limRneg}
\end{align}
Observe $\gamma \to \infty$ as $\mathbf{k}\to 1$.  
If $k_2>k_1+1$, then 
\[
\lim_{\mathbf{k}\to 1}\left(k_1+k_2+1-2(k_1+1)\mathbf{k}^2\right)=-k_1+k_2-1>0
\]
holds so that \eqref{align:R1} implies \eqref{align:limRpos}.  
%Similarly, if $k_2>k_1+1$, then $R\to\infty$ as $\mathbf{k}\to 1$.  
We obtain \eqref{align:limRneg} similarly.  
Suppose $k_2=k_1+1$.  
%\begin{align*}
%R-\check{R}&=(k_1+k_2)\left(k_1+k_2+1-2(k_1+1)\mathbf{k}^2\right)\gamma^2-a_2^2/\gamma^2\\
%&=\sqrt{\frac{k_1+k_2}{k_1+k_2+3}}\sqrt{\frac{a_1^2-(1-\mathbf{k}^2)a_2^2}{\mathbf{k}^2}}\frac{k_1+k_2+1-2(k_1+1)\mathbf{k}^2}{\sqrt{1-\mathbf{k}^2}}
%-a_2^2/\gamma^2.  
%\end{align*}
Since l'H\^opital's rule yields
\begin{align*}
\lim_{\mathbf{k}\to 1}\left(\frac{k_1+k_2+1-2(k_1+1)\mathbf{k}^2}{\sqrt{1-\mathbf{k}^2}}\right)^2
%=\lim_{\mathbf{k}\to 1}\frac{-4(k_1+1)\mathbf{k}}{-\mathbf{k}/\sqrt{1-\mathbf{k}^2}}
%&=\lim_{\mathbf{k}\to 1}\frac{-2(k_1+k_2+1-2(k_1+1)\mathbf{k}^2)4(k_1+1)\mathbf{k}}{-2\mathbf{k}}\\
=\lim_{\mathbf{k}\to 1}4(k_1+1)(k_1+k_2+1-2(k_1+1)\mathbf{k}^2)
=0, 
\end{align*}
\eqref{align:R2} implies \eqref{align:limRzero}.  
%
%This proves the theorem.  
\end{proof}

%\subsection{Proof of Lemmas ??}
%
%Let $P \to B$ be a principal $H$ bundle with connection $\omega$.  
%Suppose $H$ acts on manifolds $F_1, F_2$ and consider the associated bundle $P\times_{H} F_i$ with connection for $i=1, 2$.  
%Let $\hat{\rho}: F_1 \to F_2$ be an equivariant map.  
%The map $\id \times \hat{\rho} : P\times F_1 \to P \times F_2$ induces a map $\rho: P\times_{H} F_1 \to P\times_{H} F_2$.  
%
%\begin{lemma}
%If $\hat{\rho}$ is an injective local diffeomorphism with dense image, then so is $\rho$.  
%\end{lemma}

\begin{remark}\label{rem:parameter}
There are one or two one-parameter families of connection metrics with constant scalar curvature on $U(E_1\oplus E_2)\cong U(E_2\oplus E_1)$.  
(1) If both $\nabla^1$ and $\nabla^2$ are flat, then there is a one-parameter family of locally product metrics.  
(2) Suppose either $\nabla^1$ or $\nabla^2$ is flat but the other one is not flat.  
Interchange $E_1$ and $E_2$ if necessary and assume $\nabla^2$ is flat.  
In the notation of the previous proof, $a_1>0$ and $a_2=0$, 
and the consideration for Case 1 shows that there is a one-parameter family.  
(3) Suppose neither $\nabla^1$ nor $\nabla^2$ is flat.  
If $a_1=a_2$, then there are two one-parameter families (Case 2).  
If $a_1\neq a_2$, then, interchanging the roles of $E_1$ and $E_2$, 
we see that there are two one-parameter families (Case 1 and Case 3).  
\end{remark}

\begin{remark}\label{rem:inductive}
Let $(S^{k_1+k_2+1}, \hat{g})\to\left(U(E_1\oplus E_2), g\right)\xrightarrow{\pi}(B, \check{g})$ be the fiberwise join in Theorem \ref{Thm:FiberJoin}.  
The preceding proof implies that the following are equivalent: 
(1) $(S^{k_1+k_2+1}, \hat{g})$ is a round sphere, 
(2) the scalar curvature of $\hat{g}$ is constant, and 
(3) the norm $\lvert A\rvert$ of O'Neill tensor for $\pi$ is constant.  
%\begin{enumerate}
%\item The typical fiber $(S^{k_1}*S^{k_2}, \hat{g})$ is a round sphere.  
%\item The scalar curvature of the typical fiber $(S^{k_1}*S^{k_2}, \hat{g})$ is constant.  
%\item $\lvert A\rvert$, the norm of O'Neill tensor for the fiberwise join $\pi: \left(U(E_1\oplus E_2), g\right)\to(B, \check{g})$, is constant.  
%%\item $\mathbf{k}=0$.  
%%\item $\lvert A^1\rvert=\lvert A^2\rvert$.  
%\end{enumerate}
(2) and (3) are equivalent since the base space $(B, \check{g})$ has constant scalar curvature (see Proposition \ref{cscRiemSubm}).  
(2) implies (1) as follows.  
$\hat{g}$ has constant scalar curvature if and only if $a_1^2 f_1^2 + a_2^2 f_2^2$ is constant (see \eqref{eq:ScalJoin}, \eqref{eq:ODE}).  
Since $(1-\mathbf{k}^2)f_1^2+f_2^2=1/\gamma^2$ holds (see \eqref{align:conv}), this is possible only if 
$a_1^2 : a_2^2 = 1-\mathbf{k}^2 : 1$, 
so \eqref{eq:QEAl} yields $\mathbf{k}=0$.  
Hence $\cn_{\mathbf{k}}=\cos$ and $\sn_{\mathbf{k}}=\sin$ are trigonometric functions, 
and $\hat{g}$ has constant sectional curvature.  
%(1) The scalar curvature of the typical fiber $(S^{k_1}*S^{k_2}, \hat{g})$ is constant.  
%(2) $\mathbf{k}=0$.  
%(3) $\lvert A^1\rvert=\lvert A^2\rvert$.  
%(4) $\lvert A\rvert$, the norm of O'Neill tensor for the fiberwise join $\pi: \left(U(E_1\oplus E_2), g\right)\to(B, \check{g})$, is constant.  

In particular, if $\lvert A^1\rvert \neq \lvert A^2\rvert$, 
we cannot rescale the fiber of $\pi$ to obtain metrics of constant scalar curvature 
as in the case of Theorem \ref{Thm:irrep} (see Proposition \ref{cscRiemSubm}).  
\end{remark}

%if the scalar curvature of the typical fiber $(S^{k_1}*S^{k_2}, \hat{g})$ is constant, 
%then $\lvert A^1\rvert=\lvert A^2\rvert$.  
%$(1-\mathbf{k}^2)f_1^2+f_2^2=1/\gamma^2$
%\begin{remark}\label{rem:inductive}
%In the proof of Theorem \ref{Thm:FiberJoin}, we can take $\mathbf{k}=0$ if and only if $\lvert A^1\rvert=\lvert A^2\rvert$ holds.  
%See \eqref{eq:QEAl}.  
%
%If $\mathbf{k}=0$, then $\cn_{\mathbf{k}}=\cos$ and $\sn_{\mathbf{k}}=\sin$ are trigonometric functions, 
%so the metric $\hat{g}$ has constant sectional curvature.  
%In particular, provided $\lvert A^1\rvert=\lvert A^2\rvert$, 
%the norm $\lvert A\rvert^2=\lvert A^1\rvert^2\cos^2/\gamma^2+\lvert A^2\rvert^2\sin^2/\gamma^2$ of O'Neill tensor 
%for the fiberwise join $\pi: \left(U(E_1\oplus E_2), g\right)\to(B, \check{g})$ is constant.  
%
%However, if $\lvert A^1\rvert\neq \lvert A^2\rvert$, then $\hat{g}$ does not have constant scalar curvature 
%and $\lvert A\rvert$ is not constant.  
%In this case, we cannot rescale the fiber to obtain other metrics of constant scalar curvature 
%as in the cases of $\lvert A^1\rvert=\lvert A^2\rvert$ and of Theorem \ref{Thm:irrep} (see Proposition \ref{cscRiemSubm}).  
%\end{remark}

\subsection{Proof of Theorem \ref{Thm:sum}}\label{sect:ProofOfThm:sum}

The proof is based on Theorems \ref{Thm:irrep} %, \ref{Thm:FiberSusp}, 
and \ref{Thm:FiberJoin}.  

\begin{proof}[Proof of Theorem \ref{Thm:sum}]
%%We consider the following cases 
%In the following cases, we have Theorem \ref{Thm:irrep}.  
%$k_0\ge 1$, $l=0$ (trivial representations), $k_0=0$, $l=1$ (irreducible representations).  
%
%1. $k_0\ge 1$, $l=1$.  
%The representation $\phi=\mathbf{1}\oplus \phi_1$ is the direct sum of 
%
%2. Suppose $k_0=0$, $l=2$.  
Let $(B, \check{g})=G/H$ be a Riemannian homogeneous space 
and $\phi: H\to \SO(k+1)$ an orthogonal representation of $H$ on $\real^{k+1}$.  
%Firstly, we claim: 
%\begin{quote}
%If $\phi$ is the direct sum of at most two representations which are either trivial or irreducible, 
%then we can construct 
%connection metrics of constant scalar curvature on the unit sphere bundle $UE$ of $E=G\times_{\phi}\real^{k+1}$.  
%\end{quote}
We write $\phi=\phi_1\oplus\phi_2$ 
%We write $\phi=\psi_1\oplus\psi_2$ 
where $\phi_i$ is a trivial or irreducible representation on $\real^{k_i+1}$ for $i=1, 2$.  %, $k=k_1+k_2+1$.  
%where $\psi_i$ is a trivial or irreducible representation on $\real^{j_i}$ for $i=1, 2$.  
%Without loss of generality, we assume $j_1 \ge j_2$, $j_1+j_2\ge 2$.  
%Note that $\phi_i$ is trivial if $k_{i}+1=0$ or $1$.  

If $k_{1}+1=0$ or $k_{2}+1=0$, then $\phi$ itself is trivial or irreducible, so 
we can define connection metrics of constant scalar curvature on $U(G\times_{\phi}\real^{k+1})$ using Theorem \ref{Thm:irrep}.  

Assume $k_{1}+1\ge 1$, $k_{2}+1\ge 1$.  
Firstly, we build connection metrics on $U(G\times_{\phi_i}\real^{k_i+1})$ making use of Theorem \ref{Thm:irrep}.  
Then using Theorem \ref{Thm:FiberJoin}, we define connection metrics of constant scalar curvature on 
\[
U(G\times_{\phi}\real^{k+1})=U(G\times_{\phi_1}\real^{k_1+1}\oplus G\times_{\phi_2}\real^{k_2+1}).  
\]
This proves Theorem \ref{Thm:sum}.  
%Note that $\psi_i$ is trivial if $j_i=0, 1$.  
%\begin{itemize}
%\item If $j_1=0$ or $j_2=0$, then $\phi$ is trivial or irreducible.  Use Theorem \ref{Thm:irrep}.  
%%This is the situation in Theorem \ref{Thm:irrep}.  
%\item %Suppose $j_1=j_2=1$.  Since $\phi$ is the direct sum of trivial representations, $\phi$ is trivial.  
%If $j_1=j_2=1$, $\phi$ is trivial.  Use Theorem \ref{Thm:irrep}.  
%%We have only to consider the product metrics $\check{g}\oplus r^2\overset{\circ}{g}_1$.  
%\item If $j_1\ge 2$, $j_2=1$, construct connection metrics using Theorem \ref{Thm:FiberSusp}.  
%\item If $j_1, j_2\ge 2$, construct connection metrics using Theorem \ref{Thm:FiberJoin}.    
%\end{itemize}
%%The representation $\phi$ in this case is the direct sum of at most two representations which are either trivial or irreducible.  
%%Since we have Theorem \ref{Thm:irrep} for trivial or irreducible representations, 
%%assume $\phi=\psi_1\oplus\psi_2$
%This proves the claim, and 
%Theorem \ref{Thm:sum} for the cases 1 and 2 follows directly.  
\end{proof}

Every orthogonal representation $\phi: H\to SO(k+1)$ of a compact group $H$ is completely reducible.  
Namely, $\phi$ is equivalent to the direct sum 
$\mathbf{1}_{k_0+1}\oplus \phi_1 \oplus \dots \oplus \phi_l$ of 
the trivial representation $\mathbf{1}_{k_0+1}$ on the Euclidean space of dimension $k_{0}+1\ge 0$ 
and $l\ge 0$ irreducible representations $\phi_1, \dots, \phi_l$ of dimension $k_1+1, \dots, k_l+1\ge 2$.  

We cannot apply Theorem \ref{Thm:FiberJoin} inductively in general, 
for the typical fibers do not have constant sectional curvature if $\lvert A^1\rvert \neq \lvert A^2\rvert$ (cf.\ Remark \ref{rem:inductive}).  
However, provided $l\ge 1$ and $\phi_1=\dots=\phi_l$ in the notation above, 
we can define connection metrics of constant scalar curvature on 
\[
U(G\times_{\phi}\real^{k+1})=U\left((B\times\real^{k_0+1})\oplus G\times_{\phi_1}\real^{k_1+1}\oplus \dots \oplus G\times_{\phi_l}\real^{k_l+1}\right).  
\]
For this, we have only to apply Theorem \ref{Thm:FiberJoin} for $(l-1)$ or $l$ times according to $k_0+1=0$ or $k_0+1\ge 1$ 
since $\lvert A^0\rvert=0$, $\lvert A^1\rvert=\dots=\lvert A^l\rvert$.  
%we can apply Theorem \ref{Thm:FiberJoin} inductively 
%to obtain the connection metric $\check{g}\oplus \overset{\circ}{g}_{k_1+\dots+ k_l+l-1}$ of constant scalar curvature 
%on the unit sphere bundle of 
%\[G\times_{\phi_1\oplus\dots\oplus\phi_l}\real^{k_1+\dots+ k_l+l}=(G\times_{\phi_1}\real^{k_1+1})\oplus\dots\oplus (G\times_{\phi_1}\real^{k_1+1}).\]  
%See Remark \ref{rem:inductive}.  
%Using Theorems \ref{Thm:FiberSusp} or \ref{Thm:FiberJoin} according to $k_0=1$ or $k_0\ge 2$, 
%we can then construct metrics of constant scalar curvature on $U(G\times_{\phi}\real^{k+1})$.  This proves 3.  
It is interesting to ask if Theorem \ref{Thm:sum} holds without any assumption on the orthogonal representation $\phi$.  

%Lastly, We prove 3.  
%Suppose $\phi_1=\phi_2=\cdots=\phi_l$.  
%Since $\lvert A_1\rvert^2=\lvert A_2\rvert^2=\cdots=\lvert A_l\rvert^2$, 
%we can apply Theorem \ref{Thm:FiberJoin} inductively 
%to obtain the connection metric $\check{g}\oplus \overset{\circ}{g}_{k_1+\dots+ k_l+l-1}$ of constant scalar curvature 
%on the unit sphere bundle of 
%\[G\times_{\phi_1\oplus\dots\oplus\phi_l}\real^{k_1+\dots+ k_l+l}=(G\times_{\phi_1}\real^{k_1+1})\oplus\dots\oplus (G\times_{\phi_1}\real^{k_1+1}).\]  
%See Remark \ref{rem:inductive}.  
%Using Theorems \ref{Thm:FiberSusp} or \ref{Thm:FiberJoin} according to $k_0=1$ or $k_0\ge 2$, 
%we can then construct metrics of constant scalar curvature on $U(G\times_{\phi}\real^{k+1})$.  This proves 3.  

%\begin{lemma}
%If $\alpha_1$ and $\alpha_2$ are constants satisfying \eqref{eq:CompatConst}, 
%then the functions $f_1=...$ and $f_2=...$ satisfy \
%\end{lemma}

%\subsection{Proof of Theorem \ref{Thm:sum}}

\section{The number of constant scalar curvature metrics}\label{sect:number}

Let $(M, g)$ be an $n$-dimensional closed Riemannian manifold of constant scalar curvature.  
We consider the Yamabe equation 
\begin{align}\label{YamabePDE}
-a_n \Delta_g u +R(g) u = R(g) u^{p_n-1}, && u \in C^{\infty}_+(M)
\end{align}
with $a_n=4(n-1)/(n-2)$, $p_n=2n/(n-2)$, and $\Delta=\tr \nabla^2$.  
A function $u$ solves \eqref{YamabePDE} if and only if the conformally deformed metric $u^{p_n-2}g$ has constant scalar curvature $R(g)$.  
If $R(g)\le 0$, then the maximum principle implies that only constant functions solve \eqref{YamabePDE}.  

%\subsection{Laplacian and Riemannian submersions}
Suppose $H\overset{\text{isom}}{\lact}(F, \hat{g}) \xrightarrow{\iota} (M, g) \xrightarrow{\pi} (B, \check{g})$ 
is a Riemannian submersion with totally geodesic fibers.  
Let $C^{\infty}(F)^H$ be the subset of $C^{\infty}(F)$ consisting of functions invariant under the action $H\lact F$.  
%Through local trivializations for $\pi$ compatible with the structure group $H$, 
Through parallel transports along horizontal paths, % for $\pi$ compatible with the structure group $H$, 
we regard each $\hat{u}\in C^{\infty}(F)^H$ 
as the function $\iota_*\hat{u}$ on $M$.  
More precisely, for $y\in M$, take a horizontal path joining $y$ and a point $x\in F$, 
and define $(\iota_*\hat{u})(y)=\hat{u}(x)$.  
$\iota_*\hat{u}$ is well-defined 
%since $\Hol(HM)$ is contained in the image of the action $H\to \Diff(F)$.  
since $\hat{u}$ is invariant under the action of $\Hol(HM)$.  
The space of all such functions on $M$ is denoted by $\iota_*C^{\infty}(F)^H$.  
Similarly, $\pi^*C^{\infty}(B)$ denotes the subset of $C^{\infty}(M)$ 
consisting of functions written as $\pi^*\check{u}$ for some $\check{u}\in C^{\infty}(B)$.  

\begin{proposition}\label{prop:HorVert}
Let $u\in C^{\infty}(M)$.  
\begin{enumerate}
\item The following three conditions are equivalent.  
$u\in \pi^*C^{\infty}(B)$.  
$u$ is constant along each fiber of $\pi$.  
The gradient vector field of $u$ is normal to the fibers of $\pi$.  
\item If $u\in \iota_*C^{\infty}(F)^H$, then the gradient vector field of $u$ is tangent to the fibers of $\pi$.  
Conversely, if the gradient vector field of $u$ is tangent to the fibers of $\pi$, 
then $u\in\iota_*C^{\infty}(F)^{\Hol (HM)}$.  
\end{enumerate}
\end{proposition}
\begin{proof}
We only prove 2.  
%Note that $\grad u$ is tangent to the fibers of $\pi$ if and only if $u$ is constant along the integral curve of every horizontal vector field on $M$.  
%
Suppose $u=\iota_*\hat{u}$ for $\hat{u}\in C^{\infty}(F)^H$.  
We claim that $u$ is constant along every horizontal path.  
To show this, let $\gamma: [0, 1]\to M$ be an arbitrary horizontal path with $\gamma(0)=y_1$ and $\gamma(1)=y_2$.  
For $i=1, 2$, take horizontal paths $\gamma_i: [0, 1]\to M$ with $\gamma_i(0)=x_i$, $\gamma_i(1)=y_i$ for some $x_i \in F$.  
By definition of $\iota_*\hat{u}$, we have $u(y_i)=\hat{u}(x_i)$.  
Since $\gamma_2^{-1}\gamma\gamma_1$ is a horizontal path from $x_1$ to $x_2$ which projects onto a loop in $B$, there holds $\hat{u}(x_1)=\hat{u}(x_2)$, 
for $\hat{u}$ is invariant under $\Hol(HM)$.  
Hence $u(y_1)=u(y_2)$.  
This proves the claim.  
%since the image of $H$ under the action $H\to \Diff(F)$ is contained in $\Hol(HM)$.  
Therefore, as in the proof of Lemma \ref{HhatR}, 
we see that $\grad u$ is tangent to the fibers of $\pi$.  

Conversely, if $\grad u$ is tangent to the fibers of $\pi$, 
then $u$ is constant along the integral curve of every horizontal vector field on $M$.  
Hence $u$ is constant along every horizontal path, and $u\in \iota_*C^{\infty}(F)^{\Hol(HM)}$.  
%Note that $\grad u$ is tangent to the fibers of $\pi$ if and only if $u$ is constant along the integral curve of every horizontal vector field on $M$.  
%We only prove 2.  
%Assume $\grad u \in \Gamma (VM)$.  
%Let $\Hol(HM) \subset \Diff(F)$ is the holonomy group of $HM=(\ker d\pi)^{\perp}$ at $o\in B$ such that $\pi^{-1}(o)=F$.  So...
\end{proof}

\begin{proposition}[B\'erard-Bergery--Bourguignon \cite{BB}]
\begin{align}
\pi^*\circ\Delta_{\check{g}}=\Delta_g \circ\pi^*, &&
\iota_*\circ\Delta_{\hat{g}}=\Delta_g \circ\iota_*.  
\end{align}
\end{proposition}

\begin{corollary}\label{cor:HorVert}
Let $H\overset{\text{isom}}{\lact}(F^k, \hat{g}) \xrightarrow{\iota} (M^n, g) \xrightarrow{\pi} (B^m, \check{g})$ be a Riemannian submersion with totally geodesic fibers, 
and assume $g$ has constant scalar curvature.  
\begin{itemize}
\item Suppose $\check{u}>0$ is a smooth function on $B$.  
Then, $u=\pi^*\check{u}$ solves the Yamabe equation \eqref{YamabePDE} on $M$ if and only if $\check{u}$ solves 
\begin{align}\label{YamabeHor}
-a_n \Delta_{\check{g}} \check{u} + R(g) \check{u} = R(g) \check{u}^{p_n-1}\quad\text{on} \ B.  
\end{align}
\item Suppose $\hat{u}>0$ is an $H$-invariant smooth function on $F$.  
Then, $u=\iota_*\hat{u}$ solves the Yamabe equation \eqref{YamabePDE} on $M$ if and only if $\hat{u}$ solves 
\begin{align}\label{YamabeVert}
-a_n \Delta_{\hat{g}} \hat{u} + R(g) \hat{u} = R(g) \hat{u}^{p_n-1}\quad\text{on} \ F.  
\end{align}
\end{itemize}
\end{corollary}
\noindent If $m\ge 1$ and $k\ge 1$, %If $\dim F\ge 1$ and $\dim B\ge 1$, 
then \eqref{YamabeHor} and \eqref{YamabeVert} are subcritical in a sense that 
the respective Sobolev embeddings $W^{1, 2}\hookrightarrow L^{p_n}$ are compact.   

%It is convenient to rephrase their results in terms of the spectrum of Laplacian.  
%and consider the Riemannian product $(M, g)=(B, \check{g})\times S^k(1)$ for $m\ge 1$, $k\ge 1$, $n:=m+k\ge 3$.  
%The spectrum of $-\Delta_g$ acting on functions contains the subsequence $\{\hat{\lambda_l}=l(l+k-1)\}_{l\ge 0}$
%consisting of eigenvalues of the Laplacian for $S^k(1)$.    
%The scalar curvature $R$ of $g$ is equal to $\check{R}+k(k-1)$, 
%where $\check{R}$ is the scalar curvature of $\check{g}$.  
%Their results roughly\footnote{
%The precise statement for the $k\ge 2$ case due to Petean is the following.  
%If $\hat{\lambda}_{2l}<R/(n-1)$, 
%then $\#(g)\ge 2l+1$.   
%} show that if $\hat{\lambda}_l<R/(n-1)$, 
%or equivalently, if $l(l+k-1)(n-1)<\check{R}+k(k-1)$, 
%then $\#(g)\ge l+1$.  
%%Compare the inequality \eqref{stability}.  
%In particular, if we consider the family of metrics $g(r)=\check{g}+r^2g_{S^k(1)}$ for $r>0$, 
%then $\#(g_r)$ diverges to $\infty$ as $r\to \infty$ provided $\check{R}>0$.  
%Let $S^d(r)$ be the round sphere ($d\ge 1$, $r>0$) 
%and take a real number $R$ and an integer $n$ such that $n>\max\{d, 2\}$.  

Let $(N, h)$ be a $d$-dimensional closed Riemannian manifold.  
Take constants $N>d$, $R>0$ and consider the subcritical PDE
\begin{align}\label{YamabeTypePDE}
-a_N \Delta_h v +R v = R v^{p_N-1}, && v\in C^{\infty}_+(N).  
\end{align}
The following theorem is due to O. Kobayashi (\cite{KobO1}, \cite{KobO2}), Schoen \cite{Sch2} for $d=1$ 
and Bidaut-Veron--Veron \cite[Theorem 6.1]{VV} for $d\ge 2$.  
\begin{theorem}\label{thm:YamabeUni}
Assume $d=1$ and $\lambda_1(-\Delta_h) \ge R/(N-1)$, or $d\ge 2$ and $\Ric(h) \ge \frac{d-1}{d}\frac{R}{N-1}h$.  
Then $v\equiv 1$ is the unique solution to \eqref{YamabeTypePDE}.  
\end{theorem}
\noindent Note that $\Ric(h) \ge \frac{d-1}{d}\frac{R}{N-1}h$ implies $\lambda_1(-\Delta_h) \ge R/(N-1)$.  
O. Kobayashi (\cite{KobO1}, \cite{KobO2}), Schoen \cite{Sch2} for $d=1$ and 
Petean (\cite{Petea}, see also \cite{JLX}, \cite{HP}) for $d\ge 2$ proved the following.  
\begin{theorem}\label{thm:YamabeMult}
Assume $(N, h)=S^d(r)$ is the round sphere.  
If $\lambda_l(-\Delta_h) =l(l+d-1)/r^2< R/(N-1)$, then \eqref{YamabeTypePDE} has at least $l+1$ solutions 
invariant under the cohomogeneity-one action of $\SO(d)$ on $S^d$.  
\end{theorem}
%\noindent We recall $\lambda_l (-\Delta_{S^d(r)})=l(l+d-1)/r^2$.  

Corollary \ref{cor:HorVert}, Theorems \ref{thm:YamabeUni}, \ref{thm:YamabeMult} then imply the following.  
\begin{proposition}\label{prop:number}
Let $H\overset{\text{isom}}{\lact}(F^k, \hat{g}) \xrightarrow{\iota} (M^n, g) \xrightarrow{\pi} (B^m, \check{g})$ be a Riemannian submersion with totally geodesic fibers.  
Suppose $g$ has constant scalar curvature $R(g)>0$ and $m, k\ge 1$.  
\begin{enumerate}
\item If $\Ric(\hat{g})\ge \frac{k-1}{k} \frac{R(g)}{n-1}\hat{g}$, then the solution $u\in \iota_*C^{\infty}_+(F)^H$ to  \eqref{YamabePDE} is unique.  
\item If $H\lact (F, \hat{g}) = \SO(k) \lact S^k(r)$ and $l(l+k-1)/r^2 < R(g)/(n-1)$, 
then there are at least $l+1$ solutions to \eqref{YamabePDE} in $\iota_*C^{\infty}_+(F)^H$.  
\item If $\Ric(\check{g})\ge \frac{m-1}{m} \frac{R(g)}{n-1}\check{g}$, then the solution $u\in \pi^*C^{\infty}_+(B)$ to  \eqref{YamabePDE} is unique.  
\item If $(B^m, \check{g})=S^m(1)$ and $l(l+m-1) < R(g)/(n-1)$, 
then there are at least $l+1$ solutions to \eqref{YamabePDE} in $\pi^*C^{\infty}_+(B)$.  
\end{enumerate}
\end{proposition}
\noindent The metrics in Theorems \ref{Thm:irrep}, \ref{Thm:sum} have constant scalar curvature 
and are defined on the total spaces of Riemannian submersions with totally geodesic fibers.  
Applying Proposition \ref{prop:number}, 
we can study the number of constant scalar curvature metrics in their conformal classes.  

\begin{proof}[Proof of Theorem \ref{Thm:number}]
There hold: $(B, \check{g})=S^m(1)$, %$\Ric(\check{g})=(m-1)\check{g}$, 
%$H\lact (F, \hat{g}) = \Orth(k+1)\lact S^k(r)$, 
$(F, \hat{g}) = S^k(r)$, %$\Ric(\hat{g})=(k-1)\hat{g}$, 
and $R(g(r))=m(m-1)+k(k-1)/r^2-a^2r^2$.  

We prove (2) and (3).  
$\Ric(\check{g})\ge \frac{m-1}{m}\frac{R(g(r))}{n-1}\check{g}$ is equivalent to $-(a^2/k)r^2+(k-1)/r^2\le m$.  %$-a^2r^2+k(k-1)/r^2\le mk$.  
$l(l+m-1) < R(g(r))/(n-1)$ is equivalent to $-a^2r^2+k(k-1)/r^2>l(l+m-1)(m+k-1)-m(m-1)$.  
For a smooth function $u$ on $M$, $u\in \pi^*C^{\infty}(B)$ if and only if $u$ is constant along each fiber of $\pi$ (Proposition \ref{prop:HorVert}).   
Hence (3) and (4) in Proposition \ref{prop:number} imply (2) and (3).  %of Theorem \ref{Thm:number}.  
%1 follows similarly.  
Similarly, (1) of Proposition \ref{prop:number} implies (1) of Theorem \ref{Thm:number}.  
%
%Similarly, 1 and 2 of Proposition \ref{prop:number} imply 1 and 2 of Theorem \ref{Thm:number}.  
%1 follows similarly.  
%For the proof or 2, we have only to recall that 
%We prove 2.    
%$\Ric(\hat{g})\ge \frac{k-1}{k} \frac{R(g)}{n-1}\hat{g}$ is equivalent to $-(a^2/m)r^4+(m-1)r^2\le k$.  
%If the gradient vector field of a function $u\in C^{\infty}(M)$ is tangent to the fibers of $\pi$, 
%then $u$ can be written as $u=\iota^*\hat{u}$ for some $\hat{u}\in C^{\infty}(F)^{\Hol(HM)}$ (Proposition \ref{prop:HorVert}).  
%%possibly after reducing the structure group $H$.   
%Hence 1 of Proposition \ref{prop:number} implies 4.  
\end{proof}

%One can possibly apply 2 of Proposition \ref{prop:number} in the following situation.  
%Let $\phi$ be the direct sum of an irreducible and a $1$-dimensional trivial representations, 
%and consider the connection metrics of constant scalar curvature on $U(G\times_{\phi}\real^{k+1})$ using Theorem \ref{Thm:sum}.  
%The structure group of $U(G\times_{\phi}\real^{k+1})$ can then be reduced to $\SO(k)$ acting on $S^{k}$ with cohomogeneity one.  
%By computing the Ricci curvature (cf.\ \cite[]{Peter}), 
%we see that $\Ric(\hat{g})\ge \frac{k-1}{k} \frac{R(g)}{n-1}\hat{g}$ holds if $(m-1)n \le R$.  
%Hence 2 of Proposition \ref{prop:number} can be applicable only in the case $0< R <(m-1)n$.  
%We do not know , however.  

\begin{remark}
Some of the connection metrics that we constructed in Theorems \ref{Thm:irrep}, \ref{Thm:sum} are positive Yamabe minimizers.  
%some of the metrics in Theorems \ref{Thm:irrep}, \ref{Thm:sum} are 
%the unique unit-volume metrics of (positive) constant scalar curvature in their respective conformal classes.  
Indeed, some of the metrics in Theorems \ref{Thm:irrep}, \ref{Thm:sum} are scalar flat, 
and thus they are unique unit-volume metrics of constant scalar curvature in their conformal classes and are strictly stable with respect to the Yamabe functionals.  
Applying a slight modification of the previous result due to B\"ohm--Wang--Ziller \cite[Theorem 5.1]{BWZ}, 
we see that the connection metrics of constant scalar curvature 
close to the scalar flat ones are also Yamabe (cf.\ \cite{LPZ2}, \cite[Proposition 4]{Oto}).  
\end{remark}

%Uniqueness: Gidas--Ni--Nirenberg, B\"ohm--Wang--Ziller
\begin{remark}
The $(l+1)$ metrics of constant scalar curvature in (3) of Theorem \ref{Thm:number}
should be non-isometric to each other.  
%It is interesting to ask how different are the $l+1$ metrics of constant scalar curvature in Theorem \ref{Thm:number}.  
%We expect that none of them are isometric to each other but do not have a proof.  
%We add a heuristic argument to support our expectation.  
Assume $g(r)$ is not conformally flat and $n=m+k\ge 4$ (use the Cotton tensor instead of Weyl tensor if $n=3$).  
For $i=1, 2$, let $u_i$ be a conformal factor such that $\tilde{g}_i:=u_i^{p_n-2}g(r)$ has constant scalar curvature $R(g(r))$.  
The norm of Weyl tensor satisfies $\lvert W_{\tilde{g}_i}\rvert =u_i^{2-p_n}\lvert W_{g(r)}\rvert$.  
Since $u_i$ is constant along each fiber and $W_{g(r)}$ is invariant under the fiber-transitive action of $G$ on $UE$, 
%we obtain 
\[
\max_{UE}  \lvert W_{\tilde{g}_i}\rvert=(\min_B u_i)^{2-p_n}\max_{UE} \lvert W_{g(r)}\rvert>0.  
\]
%The gradient of $\lvert W_{g(r)}\rvert^2$ is tangent to the fibers since 
%If $f^*\tilde{g}_2=\tilde{g}_1$ holds for a diffeomorphism $f$ on $UE$, 
%then $f^*\lvert W_{\tilde{g}_2}\rvert=\lvert W_{\tilde{g}_1}\rvert$.  
If $\tilde{g}_1$ and $\tilde{g}_2$ are isometric to each other, 
%Hence $(f^*u_2)^{2-p_n}f^*\lvert W_{g(r)}\rvert=u_1^{2-p_n}\lvert W_{g(r)}\rvert$
then $\max  \lvert W_{\tilde{g}_1}\rvert=\max  \lvert W_{\tilde{g}_1}\rvert$ holds 
so that %Since $\max \lvert W_{g(r)}\rvert>0$, we obtain 
$\min u_1=\min u_2$.  
%Since $u_i$ solves a certain ODE (cf.\ \cite{Petea}), 
%it should be the case that the minimum of $u_i$ is attained at the poles of $B=S^m$, and we would obtain $u_1=u_2$.  
The $u_i$'s are obtained as solutions of ODE (see \cite{Petea}), 
and one can check that if $\min u_1=\min u_2$, then $u_1=u_2$.  
With this argument, one could see that 
the metrics in Theorem \ref{Thm:number} should be non-isometric to each other in most cases.  
However, it is difficult to rule out the possibility that some of them are isometric.  
\end{remark}

\subsubsection*{Acknowledgements}
We carried out a part of this work during the first author's two-months stay at CIMAT.  
He expresses his gratitude to the kindest hospitality of the second author and the institute.  
%After preparing \cite{OP}, 
%we completely rewrote the manuscript following sincere suggestions of the referees.  
We thank the referees of the journal for their interesting, detailed comments that helped us improve the manuscript.  
N.\ Otoba is supported by Japan Society for the Promotion of Science under Research Fellowship for Young Scientists.  
J.\ Petean is supported by grant 220074 of Fondo Sectorial de Investigaci\'on para la Educaci\'on SEP-CONACYT.  
%and acknowledges a communication on the compatibility condition \eqref{compatibility} with his director H. Izeki.  
%We thank F. Morgan for kindly pointing out the reference \cite{RCBM}.  

\end{document}